\documentclass[a4paper,11pt,leqno]{article}

\usepackage{amsmath, amsfonts} 
\usepackage{mathrsfs} 
\usepackage{amssymb, stmaryrd} 

\usepackage[dvipsnames]{xcolor}
\usepackage[modulo]{lineno}
\usepackage{todonotes}

\usepackage{graphicx}
\usepackage[all]{xy}

\usepackage{amsthm} 
\usepackage{microtype} 
\usepackage{fullpage} 
\usepackage{enumitem} 
\usepackage{enotez} 
\usepackage{authblk} 
\usepackage{setspace} 

\setenotez{backref=true}
\usepackage[backend=biber, style=alphabetic, sorting=nyt]{biblatex}
\usepackage[colorlinks=true,linkcolor=blue,citecolor=blue,urlcolor=blue,citebordercolor={0 0 1},urlbordercolor={0 0 1},linkbordercolor={0 0 1}]{hyperref} 
\usepackage{cleveref}

\usepackage{tikz-cd}

\theoremstyle{plain}
\newtheorem{thm}{Theorem}[section]
\newtheorem{lem}[thm]{Lemma}

\newtheorem{prop}[thm]{Proposition}

\theoremstyle{definition}
\newtheorem{rem}[thm]{Remark}
\newtheorem{defn}[thm]{Definition}

\NewDocumentCommand{\dan}{om}{%
  \IfNoValueTF{#1}
    {\todo[author=Dan,color=SkyBlue, bordercolor=blue, linecolor=blue]{#2}} 
    {\todo[author=Dan,color=SkyBlue, bordercolor=blue, linecolor=blue, #1]{#2}} 
}

\makeatletter 
\def\makeCal#1{%
\expandafter\newcommand\csname c#1\endcsname{\mathcal{#1}}}

\def\makeBB#1{%
\expandafter\newcommand\csname b#1\endcsname{\mathbb{#1}}}

\def\makeFrak#1{%
\expandafter\newcommand\csname f#1\endcsname{\mathfrak{#1}}}

\def\makeScr#1{%
\expandafter\newcommand\csname s#1\endcsname{\mathscr{#1}}}

\count@=0
\loop
\advance\count@ 1
\edef\y{\@Alph\count@} 
\expandafter\makeCal\y
\expandafter\makeBB\y
\expandafter\makeFrak\y
\expandafter\makeScr\y
\ifnum\count@<26
\repeat

\def\makelowercaseFrak#1{%
\expandafter\newcommand\csname mf#1\endcsname{\mathfrak{#1}}}
\count@=0
\loop
\advance\count@ 1
\edef\y{\@alph\count@} 
\expandafter\makelowercaseFrak\y
\ifnum\count@<26
\repeat

\newcommand{\category}{\mathsf} 

\DeclareMathOperator{\Set}{\category{Set}}
\DeclareMathOperator{\Sch}{\category{Sch}}
\DeclareMathOperator{\Mod}{\category{Mod}}
\DeclareMathOperator{\QCoh}{\category{QCoh}}
\DeclareMathOperator{\Ab}{\category{Ab}}

\DeclareMathOperator{\Hom}{Hom}
\DeclareMathOperator{\Aut}{Aut}
\DeclareMathOperator{\Spec}{Spec}
\DeclareMathOperator{\coker}{\mathrm{coker}} 
\DeclareMathOperator{\Ext}{\mathrm{Ext}} 

\DeclareMathOperator{\crimp}{Crimp}

\DeclareMathOperator{\supp}{Supp}

\DeclareMathOperator{\cover}{\cH}
\DeclareMathOperator{\Hilb}{Hilb}
\DeclareMathOperator{\Cond}{Cond}
\DeclareMathOperator{\Ann}{Ann}
\DeclareMathOperator{\Fitt}{Fitt}
\DeclareMathOperator{\Tor}{Tor}
\DeclareMathOperator{\SL}{SL}
\DeclareMathOperator{\Gr}{Gr}
\DeclareMathOperator{\exal}{Exal}
\DeclareMathOperator{\len}{len}
\DeclareMathOperator{\Grad}{Grad}
\DeclareMathOperator{\wt}{wt}
\DeclareMathOperator{\ev}{ev}
\DeclareMathOperator{\diag}{diag}
\DeclareMathOperator{\Filt}{Filt}

\title{Proper moduli spaces of isolated non-normal singularities}
\author{Jiucheng Dai}
\author{Daniel Halpern-Leistner}
\author{Mingjun Sun}
\author{Ni Tang}
\author{John Veliz}
\affil{Cornell University}
\date{\today}

\begin{document}

\maketitle

\begin{abstract}
    We study the moduli of isolated non-normal singularities by introducing the functor $\crimp_S^d(X)$ of crimpings of $(X \to S)$ corank $d$. This globalizes Ishii's moduli functor of subrings of finite colength. By verifying Artin's criteria, we prove that if $X \to S$ is a separated morphism of finite presentation, then $\crimp_S^d(X)$ is represented by a separated algebraic space of finite presentation over $S$, which is proper when $X \to S$ is proper. We use this to construct a semistability condition and moduli space for geometrically unibranch curves of genus zero.
\end{abstract}

\tableofcontents

\section{Introduction}

A remarkable amount of algebraic geometry and moduli theory amounts to studying the geometry of a small class of moduli spaces, which one can reformulate in terms of successively simpler moduli problems until they are ultimately described as spaces of lines in a finite dimensional vector space satisfying certain closed conditions, i.e., a closed subscheme of $\bP^n$. For instance, Hilbert schemes are special cases of Quot schemes, which can be explicitly embedded in Grassmannians $\Gr_d(H^0(X,\cO_X(k)^N))$ and ultimately in $\bP(\bigwedge^d H^0(X,\cO_X(k)^N))$. This paper constructs a new class of moduli spaces for objects of fundamental geometric interest, which apparently can not be reduced to this well-studied hierarchy.

Most moduli spaces parameterize objects ``over" a fixed space $X$, such as the Hilbert scheme $\Hilb(X)$, the Quot scheme $\rm{Quot}(X)$, the spaces of Hulls and Husks over $X$ \cite{kollár2009hullshusks}, and the space of Branchvarieties \cite{MR2608190}. In contrast, our new space, the space of \emph{crimpings}, parameterizes schemes ``under" $X$. 
A crimping of $X$ is, roughly, a finite morphism $f: X\to Y$
such that $\cO_Y \to f_* \cO_X$ is injective and $Q:= f_\ast(\cO_X)/\cO_Y$ has finite support. 
In particular, the morphism $f$ is an isomorphism away from the support of $Q$. We call the length of $Q$ the \emph{corank} of the crimping. In \Cref{D:moduli_functor}, we introduce a moduli functor $\crimp_S^d(X)$ parametrizing flat families of crimpings of $X$ of corank $d$, and our main theorem, \Cref{T:main_algebraicity}, states that this functor is represented by an algebraic space that is proper if $X$ is proper.

If $X$ is normal, then any crimping $f : X \to Y$ identifies $X$ with the normalization of $Y$. Therefore, we can regard $\crimp^d(X)$ as the moduli space of isolated non-normal singularities whose normalization is parameterized by $X$. This moduli problem of (parameterized) isolated non-normal singularities has been thoroughly studied \cite{MR3746616}, but to our knowledge this is the first time it has been observed to be representable.


When $X =\Spec A$ for a local ring $A$, the crimping functor becomes a moduli problem for subalgebras of finite colength. This perspective is natural for curve singularities, where the normalization is locally modeled by $k[[t]]$. 
Ishii first defined  the functor \cite[Definition 1,2]{MR602077}  for $X = \Spec A$ over any field $k$. She proved that, when $A$ is an Artinian local $k$-algebra, the functor is represented by a projective scheme called the territory of $A$, but representability by schemes fails if $A$ is not Artinian \cite[Theorem 1, Remark (ii)] {MR602077}.

Our project stems from the observation that the failure of representability comes from the presence of certain directions in the deformation space of the singularity that want to break it into multiple simpler singular points. These deformations can not be realized by the moduli functor for the local ring $A$, which forces the singularity to be supported at the unique closed point. Ishii addressed this issue in \cite{MR680357} by including the points on $X$ where singularities occur as part of the data, leading to a class of projective varieties that she calls relative version of territories.  We study a similar projective variety $\cover_{n,d}(X)$ below. However, our main insight is that this additional data is not necessary -- one can directly verify Artin's criteria for the algebraicity of $\crimp^d_S(X)$, as reformulated in \cite{MR3589351}.




Beyond the construction of this fundamental moduli space, our main application is to construct a proper moduli space of semistable singular curves of geometric genus $0$ and $\delta$-invariant $d$ in \Cref{SS: stability}. Unlike singular curves of geometric genus $>0$, genus $0$ singular curves can have positive dimensional automorphism groups, making this the most interesting case from a moduli-theoretic perspective. We study the stack of ``geometrically unibranch genus $0$ curves," which is a closed substack of $\crimp^d(\bP^1) / \SL_2$. This stack almost looks like a geometric invariant theory problem, except that we do not have an ample line bundle on $\crimp^d(\bP^1)$. Nevertheless, we manage to use the ``beyond GIT" framework to identify a family of numerical invariants on this stack that define a $\Theta$-stratification whose semistable locus admits a proper good moduli space (\Cref{T:genus_0}).


\subsection*{Author's note}

This paper grew out of a reading course in Fall 2025, lead by D.H.L., in which each of the authors sketched a proof of one of Artin's criteria for the stack $\crimp^d(X)$. In Spring and Summer 2026, N.T. generalized this initial work and developed it into the current paper. D.H.L. was supported by the NSF CAREER grant DMS-1945478. The use of generative AI in the this paper was limited to language editing. The paper was written by the authors, who  take full responsibility for its correctness.

\section{The moduli space of crimpings} \label{S: crimpings}

\begin{defn}\label{D:moduli_functor}
    A \emph{family of crimpings} over $T$ of corank $d$ consists of
    \begin{enumerate}
        \item finitely presented morphisms $X \to T$ and $Y \to T$,
        \item a finite morphism $f : X \to Y$ over $T$,
        \item such that the canonical morphism $\cO_Y \to f_\ast(\cO_{X })$ is injective, and the quotient sheaf $Q:=f_\ast(\cO_{X})/\cO_Y$ on $Y$ is $T$-flat and has finite support over $T$.
    \end{enumerate}
    We say that a family of crimpings is \emph{flat} if $X \to T$ is flat.
\end{defn}

The flatness of $Q$ over the base implies that, locally over $Y$, $\Tor_i(\cO_Y,M) \cong \Tor_i(f_\ast(\cO_X),M)$ for $i>0$ and any $\cO_T$-module $M$. Therefore, the flatness of $X \to T$ is equivalent to the flatness of $Y \to T$.

Given a family of crimpings over $T$ and a morphism $T' \to T$, the base change $f' : X_{T'} \to Y_{T'}$ is a family of crimpings over $T'$ -- the flatness of $f_\ast(\cO_X)/\cO_Y$ over $T$ is used to ensure that $\cO_{Y_{T'}} \to f'_\ast(\cO_{X_{T'}})$ is injective. We can therefore associate the following moduli functor to a morphism of finite presentation $X \to S$:
\begin{defn}[Moduli of crimpings]
    Let $\crimp^d_S(X) : \Sch_{/S}^{\rm{op}} \to \Set$ be the functor that assigns to an $S$-scheme $T \to S$ the set of flat families of crimpings $X \times_S T \to Y$ of corank $d$ over $T$.  We refer to these as flat families of crimpings of $X$ over $T$.
\end{defn}

\begin{thm}\label{T:main_algebraicity}
    If $S$ is a quasi-compact algebraic space and $X \to S$ is a separated morphism that is flat and of finite presentation, then $\crimp^d_S(X)$ is representable by an algebraic space that is separated and of finite presentation over $S$, which we call the space of ``crimpings'' for $X$. If in addition $X \to S$ is proper, then so is $\crimp_S^d(X) \to S$.
\end{thm}

\begin{rem}
    The hypothesis of \Cref{T:main_algebraicity} can be relaxed to the following: $X \to S$ is separated and of finite presentation and admits a universal flattening. \footnote{A universal flattening $S'_X \to S$, if exists, is always quasi-compact. If $S$ is qcqs, choose by Nagata's compactification \cite[Theorem 1.2.1]{MR2979821} an open immersion $X \to \bar{X}$ with $\bar{X} \to S$ proper of finite presentation. This gives a universal flattening $S'_{\bar{X}} \to S$ of finite presentation, and hence $S'_{\bar{X}}$ is qc. Because the natural map $S'_{\bar{X}} \to S'_X$ is surjective, this implies $S'_X$ is qc. In general if $S$ is only qc, choose an étale surjection $U \to S$ with $U$ qcqs. Then $S'_{X_U} \cong S'_X \times_{S} U$ is qc, and since $S'_{X_U} \to S'_X$ is surjective, $S'_X$ is qc. } 
    This weaker condition holds, for instance, when $X \to S$ is proper and of finite presentation \cite[\href{https://stacks.math.columbia.edu/tag/0CX2}{Tag 0CX2}]{stacks-project}. The more general formulation reduces immediately to \Cref{T:main_algebraicity}, because if $S' \to S$ is the universal flattening, then $\crimp^d_{S}(X) \cong \crimp^d_{S'}(X_{S'})$ as functors over $S'$. On the other hand, any family of crimpings $X \to Y$ of corank $0$ is an isomorphism, so $\crimp^0_S(X)(T)$ contains a unique point if $X \times_S T$ is $T$-flat, and is empty otherwise. In other words the statement of \Cref{T:main_algebraicity} in the case $d=0$ is equivalent to the existence of a universal flattening for $X \to S$.
\end{rem}


 We recall that given a family of  crimpings $f : X \to Y$, the annihilator of $Q$ in $\cO_Y$ is also a sheaf of ideals of $f_\ast(\cO_X)$, and can therefore canonically be regarded as a sheaf of ideals $\Ann(Q) \subset \cO_X$ on $X$, called the conductor ideal of the morphism $f$. We call the closed subspace $\Cond(f) \hookrightarrow X$ associated to the ideal $\Ann(Q)$ the \emph{conductor subspace} of $f$. For a family of crimpings, $\Cond(f)$ is finite over $T$, although it need not be flat even when $X_T \to T$ is.

Given a closed subspace $Z \subset X$ and an affine morphism $Z \to Z'$, the pushout $Y = X \cup_{Z} Z'$ exists in the category of algebraic spaces. Furthermore, $Z = X \times_Y Z'$, and the formation of the pushout is compatible with base change along a smooth morphism $T \to Y$ \cite[Theorem 4.2, Lemma 4.3]{MR4699880}.
Any family of crimpings $f : X \to Y$ can be recovered as the pushout
\begin{equation}
\xymatrix{\Cond(f) \ar[r] \ar[d] & \Spec_{Y}(\cO_Y/\Ann(Q)) \ar[d] \\ X \ar[r] & Y}.
\end{equation}
In fact, $\Cond(f) \hookrightarrow X $ is the smallest closed subspace that can be used to construct $f$ in this way. 
\begin{rem}Let $X \to Y$ be a family of crimpings over $T$. 
    Since both $\Cond(f)$ and $\Spec_{Y}(\cO_Y/\Ann(Q))$ are finite over $T$, it follows from the pushout construction and its formation that if $X$ is affine over $T$, then $Y$ is also affine over $T$.
\end{rem}

Crimpings are ``Nisnevich local" in the following sense:

\begin{lem}\label{L:locality_of_crimpings}
    Suppose that $X \to T$ is of finite presentation, and $p : U \to X$ is an \'etale morphism of finite presentation, and $Z \hookrightarrow X$ is a closed subscheme such that $p^{-1}(Z) \to Z$ is an isomorphism. Then there is a bijection between families of crimpings of $X$ with conductor contained in $|Z|$, and families of crimpings of $U$ with conductor contained in $|p^{-1}(Z)|$, given by the following construction: 
    
    Given $f : U \to V$ such that $|\Cond(f)| \subset |p^{-1}(Z)|$, the composition $\Cond(f) \hookrightarrow U \to X$ is also a closed immersion, and we define a crimping $X \to X \cup_{\Cond(f)} \Spec_V(\cO_V/\Ann(Q))$. 
\end{lem}
\begin{proof}
    For any closed subspace $W\subset U$, the composition $W \to X$ is unramified because $U \to X$ is \'etale, and it is injective and closed because $|p^{-1}(Z)| \to |Z|$ is a homeomorphism. Likewise, any closed subscheme of $X$ contained set-theoretically in $|Z|$ lifts uniquely to a closed subscheme of $U$, because $p$ identifies the formal completion of $X$ along $Z$ with the formal completion of $U$ along $p^{-1}(Z)$. The fact that the construction in the lemma gives a bijection between crimpings of $U$ with conductor in $|p^{-1}(Z)|$ and crimpings of $X$ with conductor in $|Z|$ follows from the discussion above the lemma, which explains that any crimping is obtained canonically as the extension of a crimping of its conductor subspace.
\end{proof}

Let $S$ be a quasi-compact algebraic space, and let $X \to S$ be of finite presentation. It will be convenient to reduce proofs of various claims to the case when $T$ is Noetherian, using the following:

\begin{lem}[Limit preservation]\label{lem: limit preservation}
    Suppose that $T = \lim_i T_i$ is an inverse limit of affine $S$-schemes, and $f : X_T \to Y$ is a (flat) family of crimpings over $T$. Then for some $i$, there is a (flat) family of crimpings $X_{T_i} \to Y_i$ over $T_i$ and an isomorphism $Y_i \times_{T_i} T \cong Y$ commuting with the maps from $X_T$ and to $T$.
\end{lem}
\begin{proof}
    Because $Y$ and $X_T$ are of finite presentation over $T$, one can replace the indexing set with a cofinal subset with initial index $0$ and find a $Y_0$ (flat and) of finite presentation over $T_0$ and a morphism $f_0 : X_{T_0} \to Y_0$ whose base change to $T$ is isomorphic to $f$ \cite[Tag 0EY1, Tag 02JO(3)]{stacks-project}. So it suffices to show that if $X_{T_0} \to Y_0$ is a morphism of spaces of finite presentation over $T_0$, and if $X_T \to Y_0 \times_{T_0} T$ is a (flat) family of crimpings, then so is $X_{T_i} \to Y_i := Y_0 \times_{T_0} T_i$ for $i \gg 0$. For $i \gg 0$, the morphism $X_{T_i} \to Y_i$ is finite by \cite[Tag 01ZO]{stacks-project}, so we may assume that for $f_0 : X_{T_0} \to Y_0$ as well.

    For all $i$, let $Q_i$ denote the finitely presented quasi-coherent sheaf $\coker(\cO_{Y_i} \to (f_i)_\ast(\cO_{X_{T_i}})) \in \mathrm{QCoh}(Y_i)$, and note that for $i \to j$, $Q_j$ is the pullback of $Q_i$ along $Y_j \to Y_i$. We claim that for $i\gg 0$, $Q_i$ has finite support over $T_i$: Indeed, let $Z_i \hookrightarrow Y_i$ be the closed subspace defined by the fitting ideal $\Fitt_0(Q_i)$. Then $Z_i \cong Z_0 \times_{T_0} T_i$ for all $i$, and $\lim_i Z_i \hookrightarrow Y$ is the closed subspace defined by $\Fitt_0(Q)$, which is finite over $T$ by hypothesis. It follows that $Z_i$ is finite over $T_i$ for $i \gg 0$ \cite[Tag 01ZO]{stacks-project}.

    Next we claim that as an $\cO_{T_i}$ module, i.e., after pushforward along $Y_i \to T_i$, $Q_i$ is locally free of finite rank for all $i\gg 0$: Indeed, $Q_i$ is finitely presented on $Y_i$ and scheme theoretically supported on $Z_i$, hence it is finitely presented on $Z_i$ as well. It follows that for $i \gg 0$, the pushforward of $Q_i$ along the finite morphism $Z_i \to T_i$ is finitely presented in $\QCoh(T_i)$. This implies that $Q_i$ locally free for $i \gg 0$ because $Q \cong Q_0 \otimes_{\cO_{Y_0}} \cO_Y$ is locally free of finite rank \cite[Tag 02JO(3)]{stacks-project}.

    The long-exact sequence for $\Tor$ implies that if $Q_i$ is $T_i$-flat, then the formation of the kernel $K_i := \ker(\cO_{Y_i} \to f_\ast(\cO_{X_{T_i}}))$ commutes with base change along any morphism to $T_i$. This observation, combined with absolute Noetherian approximation of $T$, implies that $K_i$ is finitely presented over $Y_i$ for $i\gg 0$. Finally, the fact that $f$ is a crimping implies that $K_i \otimes_{\cO_{T_i}} \cO_T \cong 0$ for $i$ large enough that $Q_i$ is $T_i$-flat, and it follows that $K_i \cong 0$ for $i\gg0$. This shows that $\cO_{Y_i} \to (f_i)_\ast(\cO_{X_{T_i}})$ is injective, and hence $f_i$ is a crimping, for $i \gg 0$.
\end{proof}

\subsection{The pinching morphism}
\label{SS: pinching}
\begin{defn}
Given a separated and finitely presented morphism $X \to S$, we define the functor $\cover(X)_{n,d} : \Sch_{/S}^{\rm{op}} \to \Set$ to take a scheme $T$ to the set of diagrams
\begin{equation}\label{E:hilbert_point}
    \xymatrix{Z \ar[r] \ar[d] & X_T \\ W & },
\end{equation}
where $Z \to X_T$ is a closed immersion, the schemes $Z$ and $W$ are finite and flat over $T$, $\cO_Z$ is localy free of rank $n$ over $T$, the morphism $\cO_{W} \to \cO_Z$ in $\QCoh(T)$ is injective, and $\cO_Z / \cO_{W}$ is locally free of rank $d$.
\end{defn}

\begin{lem}\label{lem: representability_of_H}

    The functor $\cover(X)_{n,d}$ is represented by a projective algebraic space over $\Hilb^n_X$, which itself is representable by a separated algebraic space of finite presentation over $S$ by \cite[Theorem 4.4]{MR2821738}, given $X \to S$ is separated and of finite presentation.
\end{lem}

\begin{proof}
We show $\cover(X)_{n,d}$ is represented by a closed subspace of a relative Grassmannian over $\Hilb^n_X$. 
Let $Z_{\mathrm{univ}} \subset X \times_S \Hilb^n_X $ be the universal closed subspace. 
Then $Z \to \Hilb^n_X$ is finite and flat, and its structure sheaf $\cO_{Z_{\rm{univ}}}$ is a locally free sheaf of algebras of rank $n$ over $\Hilb^{n}_{X}$. Then the Grassmannian $G:= \mathrm{Gr}_{d}(\mathcal{O}_{Z_{\mathrm{univ}}})$ is representable by a  projective algebraic space over $\Hilb^{n}_{X}$.
After pulling back $\mathcal{O}_{Z_{\mathrm{univ}}}$ to $G$, write it as $\mathcal{O}_{Z_G}$. There is a universal exact sequence $ 0 \to K_{\mathrm{univ}} \to \mathcal{O}_{Z_G} \to Q_{\mathrm{univ}} \to 0$ of coherent $\mathcal{O}_G$-sheaves on $G$, with $Q_{\mathrm{univ}} $ locally free of rank $d$ over $G$. We restrict to the smallest closed subspace  of $G$ where $K_{\mathrm{univ}}$ is an $\mathcal{O}_G$-subalgebra of $\mathcal{O}_{Z_G}$ \cite[Theorem 1]{MR602077}, which we call $G^{\mathrm{alg}}$.  

A $T$-point of $G^{\mathrm{alg}}$ is given by a closed immersion $Z \hookrightarrow X_T$, such that $Z \to T$ is finite and flat, $\mathcal{O}_Z$ is locally free of rank $n$ over $T$, together with a rank $d$ quotient $\mathcal{O}_Z \twoheadrightarrow Q$ whose kernel $K$ is an $\mathcal{O}_T$-subalgebra of $\mathcal{O}_Z$. Setting $W: = \Spec_T K$, we recover a $T$-point of $\cover(X)_{n,d}$. 
Conversely, given a diagram \eqref{E:hilbert_point} over $T$, then $\mathcal{O}_W \to \mathcal{O}_Z$ is injective, and the quotient $Q:= \mathcal{O}_Z / \mathcal{O}_W$ is locally free of rank $d$ over $T$. Since $\mathcal{O}_W$ is an $\mathcal{O}_T$-subalgebra of $\mathcal{O}_Z$, the diagram defines an element in the set $G^{\mathrm{alg}}(T)$.   
Therefore, we conclude that  $\cover(X)_{n,d}$ is represented by $G^{\mathrm{alg}}$.  


\end{proof}


\begin{prop} \label{prop: pushout}
    If $Z \to W$ is a family of crimpings over $T$, $Z \hookrightarrow X$ is a closed immersion, and $X$ is flat and finitely presented over $T$, let $Y := X \cup_Z Y$ be the pushout in the category of quasi-separated algebraic spaces. Then the canonical morphism $X \to Y$ is a \emph{flat} family of crimpings over $T$.
    
    In particular the assignment $(Z \hookrightarrow X_T, Z \to W) \mapsto (X_T \to X_T \cup_Z W)$ defines a morphism of functors $\cover(X)_{n,d} \to \crimp_{S}^d(X)$.
    
\end{prop}

\begin{proof}

We define this morphism locally over $\Spec R \subset T$. The pushout $Y = X \cup_{Z} W$ exists and is an algebraic space over $\Spec R$, and the map $X \to Y$ is affine \cite[Theorem 1.8]{MR4699880}. We verify that it is a family of crimpings of $X$ by working étale locally over $Y$. Since the formation of pushout is compatible with étale base change, it suffices to work after replacing $Y$ by an affine étale cover. Hence we may assume $Y = \Spec C$. 
Then $Y$ is a pushout of affine schemes. Indeed, 
since $X \to Y$ is affine, we may write $X = \Spec A$. Let $I$ denote the ideal defining the closed subscheme $Z \hookrightarrow X$. Then $Z =\Spec A/I$. Because $Z \to W$ is a family of crimpings, $W = \Spec B$, and $B \subset A/I$ is a $R$-subalgebra.
Therefore, we have a fibered product square in the category of $R$-algebras
\[
\begin{tikzcd}
C \ar[r] \ar[d] & A \ar[d, "\pi"] \\ B \ar[r,hook] & A/I.
\end{tikzcd}
\]

Given that $B \to A/I$ is injective, it follows that $C \cong \pi^{-1}(B)$. By definition, the map $\pi^{-1}(B) \subset A$ is a subalgbra.  We proceed to verify that $\pi^{-1}(B) \to A$ is a crimping. 
The quotient $ Q = A / \pi^{-1}(B) \cong (A/I)/ B$ is a locally free $R$-module of rank $d$. Note that $\pi^{-1}(B)$ fits in the short exact sequence as $R$-modules 

$$0 \to \pi^{-1}(B) \to  A \to A/( \pi^{-1}(B)) \to 0. $$ 

We know $A/( \pi^{-1}(B))$ is $R$-flat since it is locally free over $R$, and $A$ is $R$-flat by assumption. It follows that $\pi^{-1}(B)$ is flat over $R$. Recall that $A/I$ is finitely generated as a $B$-module; by lifting the $B$-module generators of $A/I$ to $A$, and adjoining $1$, we deduce that $A$ is finitely generated as a $\pi^{-1}(B)$-module. Assume that $R$ is Noetherian. 
Since $A$ is of finite type over $R$ and finite as a module over $\pi^{-1}(B)$, the Artin-Tate lemma  implies that $\pi^{-1}(B)$
is of finite type, and hence of finite presentation over $R$.  

Finally, we can apply standard approximation techniques to remove the assumption that $R$ is Noetherian. Let $D = A/I$. Then $B \subset D$ is  a family of crimpings over $R$. By \Cref{lem: limit preservation}, there exist a Noetherian subring $R' \subset R$, and a family of crimpings $B' \subset D'$ over $R'$ whose base change to $R$ is identified with $B \subset D$. After replacing $R'$ with a larger Noetherian subring of $R$, we may also assume there is an $R'$-algebra $A'$ over $R'$, together with a surjection $A' \to D'$ whose base change is identified with the surjection $A \to D$. It remains to show that the canonical homomorphism $(A' \times_{D'} B') \otimes_{R'} R \to A \times_D B \cong \pi^{-1}(B)$ is an isomorphism. 
This follows from the  $R'$-flatness of the quotient $Q' = D'/B'$. The short exact sequence 
\[
0 \to  A' \times_{D'} B' \to A' \to Q' \to 0,
\] 
remains exact after applying $(-)\otimes_{R'} R$ by $R'$-flatness of $Q'$. Since $A'\otimes_{R'} R \cong A$, and $Q' \otimes_{R'} R \cong Q$, we obtain $(A' \times_{D'} B') \otimes_{R'} R \cong \ker ( A\twoheadrightarrow Q )$. But $\ker ( A\twoheadrightarrow Q ) = A \times_D B$. Hence $(A' \times_{D'} B') \otimes_{R'} R \cong A \times_D B$. 
\end{proof}

\begin{prop} \label{P: boundness}
The morphism $\cover(X)_{ d(d+1), d} \to \crimp^d_S (X)$ is surjective on points over any field. \end{prop}

\begin{proof} Let $f: X_{k} \to Y$ be a $k$-point of $\crimp_S^d(X)$. In the preliminaries we saw that the flat crimping $X_{k} \to Y$ can be recovered from the  pushout of the diagram $[\Spec_{Y}(\mathcal{O}_{Y}/\Ann(Q) ) \leftarrow \mathrm{Cond}(f) \hookrightarrow X_k]$. The flatness of $\mathrm{Cond}(f)$ and $\Spec_{Y}(\mathcal{O}_{Y}/\Ann(Q) )$ is automatic since we are working over a field. The diagram above defines a point in $\cover(X)_{n, d}$, where $n$ is the colength of the conductor ideal. Our goal is to provide an upper bound for $n$. We have a short exact sequence

\[
	0 \to {\mathcal{O}_{Y}/\Ann(Q)} \to  {f_*\mathcal{O}_{X_k}/\Ann(Q)} \to  Q \to 0.
 \] 

It follows that $\dim_{k}({\mathcal{O}_{X_k}/\Ann(Q)}) = \dim_{k}({\mathcal{O}_{Y}/\Ann(Q)}) + \dim_{k}(Q)$. Since the annihilator is precisely the kernel of the map $\cO_{Y} \to \mathrm{End}_{k}(Q)$ given by multiplication, it follows that $\dim_{k}(\cO_{Y}/ \Ann(Q)) \leq \dim_k \mathrm{End}_{k}(Q) = (\dim_{k}(Q))^{2} =d^2$.  We conclude that $\dim_k({\mathcal{O}_{X}/\Ann(Q)}) \leq d(d+1)$. \end{proof}
\begin{rem}\label{R: curve_bound}
    If $X$ is a smooth curve over a field $k$, then $ d +1 \leq \dim_k(\mathcal{O}_X /\Ann(Q)) \leq 2d$ by \cite[Chapter IV, \textsection 3, Proposition 7]{MR918564}. Therefore instead of letting $n = d(d+1)$, we may take the sharper bound $2d$.  
\end{rem}

\subsection{Deformation theory of crimpings} 
\label{SS: deformation}
Let $S$ be an excellent scheme, and let $X \to S$ be a separated morphism of finite presentaion.
Throughout the discussion of the deformation theory, $T =\Spec R$ will be an affine $S$-scheme, locally of finite type over $S$. We denote the category of all square-zero extensions of $T$ over $S$ by $\exal_S(T)$. Given the assumption that $T$ is affine, any such extension $T \hookrightarrow T'$ is also affine. The kernel of $\mathcal{O}_{T'} \twoheadrightarrow \mathcal{O}_T$ is canonically a quasi-coherent $\mathcal{O}_T$-module. Let us denote $T' = \Spec R'$.  
We write $\exal_S(T,J)$ for the fiber category of $\exal_S(T)$ consisting of square-zero extensions of $T$ by $J \in \QCoh(T)$. 

Let $f:X_T \to Y$ be a flat family of crimpings over $T$. We consider the problem of extending $(f: X_T \to Y)$ over a square-zero extension $T \hookrightarrow T'$. The set of such extensions is denoted by 
\[
\exal_f (T'):= \{f': X_{T'} \to Y' \in \crimp_S^d(X)(T') \mid f \cong f' \times_{T'} T: X_T \to Y' \times_{T'} T\}.
\]
\begin{lem}\label{L: local_def}
    Let $X = \Spec A_0$ be a \emph{flat} affine scheme over $S$, and let $f: X_T \to Y:= \Spec B$ be a flat family of crimpings of corank $d$ over $T$. Suppose that $J \in \QCoh(T)$ is a quasi-coherent module on $T$. For arbitrary square-zero thickening $T \hookrightarrow T'$ with ideal $J$, 
    \begin{enumerate}
        \item there is a canonical obstruction class 
        \[
        \alpha \in \Ext^1_B(L_{B/R}, J \otimes_R Q), 
        \]
        whose vanishing is equivalent to the existence of a flat extension of $f$ over $T'$. 
        \item If $\alpha$ vanishes, then the set of such flat extensions $\exal_f(T')$ is principal homogeneous  under $\Ext^0_B (L_{B/R}, J \otimes_R Q) \cong \Hom_B (\Omega_{B/R}, J \otimes_R Q)$.
        
    \end{enumerate}
\end{lem}
A special case of part two where $T' =\Spec k[\epsilon]/(\epsilon^2)$ is treated in \cite[Theorem 1]{MR447252}, and \cite[Proposition 1]{MR602077}.
\begin{proof}
    Using \Cref{lem: limit preservation}, we may assume $J$ is coherent.
    Let $A:= A_0 \otimes_S R$, $A':= A_0 \otimes_S R'$. Define $\exal_{A,B} (T')$ to be the set of all flat $R'$-subalgebras $B' \subset A'$ such that $B' \otimes_{R'} R \cong B$. Then by definition, given an element $(f': \Spec A' \to Y') \in \exal_{f}(T')$, $\Gamma(Y',\mathcal{O}_{Y'}) \in \exal_{A,B}(T')$.  On the other hand, if $B' \in \exal_{A,B} (T')$, then we show $\Spec A' \to \Spec B' \in \crimp_S^d(X)(T')$, and therefore provides an extension of $f$ over $T'$.  Because $B'$ is flat over $T'$, tensoring the sequence $0\to J \to R' \to R \to 0$ with $B'$,  we have $0 \to J \otimes_R B \to B' \to B \to 0$. Let $Q'$ denote the quotient $A'/B'$. Then $Q'/JQ' \cong \coker(B'/JB' \to A'/JA') = \coker(B \to A) =Q$. Applying the snake lemma to
    \[
\begin{tikzcd}
0 \arrow[r] & J\otimes_{R} B \arrow[r] \arrow[d] & B' \arrow[r] \arrow[d] & B \arrow[r] \arrow[d] & 0 \\
0 \arrow[r] & J\otimes_{R} A \arrow[r]  & A' \arrow[r]  & A \arrow[r]  & 0, 
\end{tikzcd}
\]
    because all three vertical maps are injective, we obtain a short exact sequence $0 \to J \otimes_R Q \to Q' \to Q \to 0$. Hence by the local criterion for flatness, $Q'$ is flat over $T'$. Moreover, by the Nakayama's lemma, 
    $Q'$ is finite and therefore locally free of rank $d$ over $T'$. Now $A'$ is finitely generated as a $B'$-module. Because $J$ is coherent, the extension $R'$ remains Noetherian, by the Artin-Tate lemma, the subalgebra $B'$ is finitely generated over $R'$. This proves that $\Spec A' \to \Spec B'$ is a flat family of crimpings of corank $d$ over $T'$. Therefore, the described relation gives a one-to-one correspondence between two extension problems. 

    Let us focus on $\exal_{A,B} (T')$. Let $\pi$ denote the quotient map $A' \to A$, and let $E$ be the sublagebra of $A'$ given by the preimage $\pi^{-1}(B)$. Then $E$ maps onto $B$, with kernel $J \otimes_R A$. Because $Q$ is flat as a $T$-module, $J \otimes_R B \subset J \otimes A \subset E$. Furthermore, $J \otimes_R B$ is an ideal of $E$:   if $x \in E$, then $x$ acts on $J \otimes_R B$ through $\pi(x) \in B$, by multiplying  on the second tensor factor. Hence $x \cdot (J \otimes_R B) \subset J \otimes_R B$.  Noting that $(J \otimes_R A) /(J \otimes_R B) = J \otimes_R Q$,  quotienting out by $J \otimes_R B$ gives a short exact sequence \begin{equation}
        0 \to J \otimes_R Q \to E/ (J \otimes_R B) \to B \to 0. 
    \end{equation}
    A section $s: B \to E/ (J \otimes_R B)$ determines an element of $\exal_{A,B}(T')$ as follows. Let $p: E \to E/ (J \otimes_R B)$ denote the quotient map . Given such a section $s: B \to E/ (J \otimes_R B)$, set $B' = p^{-1}(s(B))$. Then $B'$ is a $R'$-subalgebra of $A'$, and $B' \otimes_{R'} R \cong B'/ (J \otimes_R B) \cong B$. Therefore, by the local criterion for flatness, $B'$ is flat over $R'$, and $B'$ gives a flat extension of $B$.  Conversely, if $B' \in \exal_{A,B}(T')$ is a flat extension, then $\pi(B') \subset B$ and therefore $B'\subset E$. Moreover, $B'/(J \otimes_{R} B) \cong B$, and this isomorphism defines a section $s: B \xrightarrow{\sim} B'/(J \otimes_{R} B) \hookrightarrow E/ (J \otimes_R B)$. 

    An extension of $B$ by $J \otimes_R Q$ over the square-zero thickening $T \hookrightarrow T'$ defines a canonical class $\alpha \in \Ext^1_B (L_{B/R}, J \otimes_R Q)$ by \cite[Tag 08UX] {stacks-project}. The vanishing of $\alpha$ is equivalent to the existence of a section of $E/(J \otimes_R B) \twoheadrightarrow B$. This establishes part one. For part two, again by \cite[Tag 08UX] {stacks-project}, given the vanishing of the obstruction class, the set of automorphisms of the trivial extension is principal homogeneous under $\Ext^0_B(L_{B/R}, J \otimes_R Q) \cong \Hom_B(\Omega_{B/R}, J \otimes_R Q)$. Since the set of automorphisms can be identified with the set of sections of $E/(J \otimes_R B) \twoheadrightarrow B$, we conclude part two.       
\end{proof}

Assume that $J \in \QCoh(T)$. Let $T[J] := \Spec (R \oplus J)$ denote the trivial square-zero thickening by $J$. We define the deformation functor at a flat family of crimpings $(f: X_T \to Y) \in \crimp_S^d(X)(T)$ as $\mathrm{Def}_{f}: \QCoh (T) \to \Ab, J \mapsto \exal_f (T[J])$. 
\begin{prop} \label{P: deformation}
Let $X \to S$ be flat, and 
    let $f: X_T \to Y$ be a flat family of crimpings of corank $d$ over $T$. Then the deformation functor at $f$ is given by 
    \[
    \mathrm{Def}_f (J) = \Ext^0_{\mathcal{O}_Y} (L_{Y/T}, J \otimes_T Q). 
    \] 
\end{prop}
\begin{proof}
    Since $f$ is a finite morphism, affine-locally on $Y$, $\exal_f(T[J])$ is provided by \Cref{L: local_def}. Because the extension is compatible under localization, the affine-local descriptions glue to give the global result. 
\end{proof}

\begin{prop} \label{P: obstruction}
Let $X \to S$ be flat. 
 Let $f: X_T \to Y$ be a flat family of crimpings of corank $d$ over $T$, and let $T \hookrightarrow T'$ be any square-zero thickening by $J \in \QCoh(T)$.  There is a canonical element $$\delta \in \Ext^1_{\mathcal{O}_Y} (L_{Y/T}, J \otimes_T Q)),$$ whose vanishing is equivalent to the existence of a deformation of $f$ over $T'$.
\end{prop}
\begin{proof}
    For each affine open $U_i \subset Y$, the obstruction to the existence of a local extension is given by a canonical element  $ \alpha_{U_i} \in \Ext^1_{\mathcal{O}_{U_i}}(L_{U_i/T}, J \otimes_T Q|_{U_i})$. These glue together to give a global obstruction $\alpha \in H^0(Y, \mathcal{E}xt^1_{\mathcal{O}_Y} (L_{Y/T}, J \otimes_T Q))$, because $\mathcal{E}xt^1(L_{Y/T}, J \otimes_T Q)$ is equivalent to the sheafification of the presheaf that assigns $(U \subseteq Y) \mapsto \Ext^1_U(L_{Y/T}|_U, J \otimes_T Q)$. Note that the quasi-coherent sheaves $\mathcal{E}xt^q_{\mathcal{O}_Y} (L_{Y/T}, J \otimes_T Q)$ are scheme-theoretically supported on $\Spec_Y (\cO_Y/\Ann(Q))$. Since $\Spec_Y (\cO_Y/\Ann(Q))$ is finite over $T$, and $T$ is an affine scheme, $\Spec_Y (\cO_Y/\Ann(Q))$ is affine. Hence these sheaves have no higher cohomology.  The five-term exact sequence associated to the Ext spectral sequence $E_2^{p,q} = H^p(Y, \mathcal{E}xt^q_{\mathcal{O}_Y}(L_{Y/T}, J \otimes_T Q)) \Rightarrow \Ext^{p+q}_{\mathcal{O}_Y}(L_{Y/T}, J \otimes_T Q)$ gives an isomorphism $H^0(Y, \mathcal{E}xt^1_{\mathcal{O}_Y} (L_{Y/T}, J \otimes_T Q)) \cong \Ext^1_{\mathcal{O}_Y} (L_{Y/T}, J \otimes_T Q)$. We define the image of $\alpha$ under this isomorphism to be the first obstruction $\delta \in \Ext^1_{\mathcal{O}_Y} (L_{Y/T}, J \otimes_T Q)$.

    If this obstruction vanishes, for each $U_{ij} = U_i \cap U_j$, the local extensions of $f_{U_i}$ and $f_{U_j}$ restrict to two extensions of $f|_{U_{ij}}$. Their difference defines an element $d_{ij} \in \Ext^0_{\mathcal{O}_{U_{ij}}} (L_{U_{ij}/T}, J \otimes_{T}Q |_{U_{ij}})$. On every triple overlap $U_{ijk}= U_i \cap U_j \cap U_k$, these elements satisfy $d_{ij} + d_{jk}- d_{ik} =0$, and therefore define  a Čech 1-cocycle $\gamma \in H^1( Y, \mathcal{E}xt^0 _{\mathcal{O}_Y}  (L_{Y/T}, J \otimes_T Q))$. This is the second obstruction.   Once it vanishes, we can modify the local extensions so that they agree on overlaps. Again since the quasi-coherent sheaf $\mathcal{E}xt^0 _{\mathcal{O}_Y}  (L_{Y/T}, J \otimes_T Q)$ is supported on an affine closed subscheme,  $H^1( Y, \mathcal{E}xt^0 _{\mathcal{O}_Y}  (L_{Y/T}, J \otimes_T Q)) =0$. It follows that the second obstruction vanishes. This completes the proof.
\end{proof}
Given a flat family of crimpings $f: X_T \to Y$, the obstruction theory $(\mathrm{Ob}_f, \mathrm{ob}_f)$ at $f$ is defined as follows. 
Set 
\[
\mathrm{Ob}_f: \QCoh(T) \to \Ab, \quad J \mapsto \Ext^1_{\mathcal{O}_Y} (L_{Y/T}, J \otimes_T Q)
\]
as an additive functor of $\QCoh(T)$ to the category of abelian groups, defining the obstruction space. 
Because of the functoriality of the canonical element given by \Cref{P: obstruction}, the assignment 
\[
(T \hookrightarrow T') \mapsto \delta 
\]
defines a natural transformation $\mathrm{ob}_f: \exal_S (T, -) \Rightarrow \mathrm{Ob}(-)$. 

\begin{prop}\label{P: coherence_of_def_obs} Let $X \to S$ be flat. 
   $\crimp_S^d(X)$ admits a coherent deformation theory and a coherent obstruction theory at any flat family of crimpings $(f: X_T \to Y) \in \crimp_S^d(X)(T)$, given respectively by $\mathrm{Def}_f$, and $(\mathrm{Ob}_f, \mathrm{ob}_f)$.  
\end{prop}
\begin{proof}
    It suffices to show that the functors $\mathrm{Def}_f, \mathrm{Ob}_f: \QCoh(T) \to \Ab$ are coherent. Since $Y$ admits a finite surjective morphism from a scheme $X_T$, and $X_T$ is separated over $T$ , it follows that $Y$ is separated over $T$. Moreover, $Q$ is properly supported and flat over $T$. Therefore, by  \cite[Theorem C]{MR3267585}, both functors are coherent.  
\end{proof}

\subsection{The main representability theorem} \label{SS: main_thm}





The proof of \Cref{T:main_algebraicity} will use Artin's criteria, as will be explained after the following preliminary lemmas. Assume that $S$ is an excellent scheme, and $X \to S$ is separated and of finite presentation.

\begin{lem}(Effectivity)
\label{L: effectivity}
For a complete local Noetherian $S$-algebra $(C,\mathfrak{m})$, the restriction map is a bijection
\[ 
\mathcal{G}: \crimp_S^d(X) (C) \xrightarrow{\cong} \varprojlim_{n} \crimp_S^d (X) ( C/\mathfrak{m}^n). 
\]
\end{lem}
\begin{proof}
We denote $C_n:= C/\mathfrak{m}^n, X_{C} := X \times_S \Spec C$, and $X_{C_n}:= X \times_S \Spec C_n$. An element of $\varprojlim_{n} \crimp_S^d(X) ( C/\mathfrak{m}^n)$ is given by a compatible system of flat families of crimpings $\{(f_n: X_{C_n} \to Y_n)\}_{n\geq 1}$, and for each flat family of crimpings $f_n: X_{C_n} \to Y_n$ over $C_n$, we have an exact sequence of $\mathcal{O}_{Y_n}$-modules: $0 \to \mathcal{O}_{Y_n} \to f_{n,*} \mathcal{O}_{X_{C_n}} \to Q_n \to 0.$ Let $W_n \hookrightarrow Y_n$ be the closed subspace cut out by the fitting ideal $\Fitt_0 (Q_n)$, and let $Z_n: = W_n \times_{Y_n} X_{C_n}$, which are both finite over $C_n$ by definition. Moreover, the short exact sequence above reduces to a short exact sequence on $W_n$:
\begin{equation}\label{equ: ses_on_fitting_ideal}
0 \to \mathcal{O}_{W_n} \to f_{n,*} \mathcal{O}_{Z_n} \to Q_n \to 0.    
\end{equation}
Since the system $\{Y_n\}$ is compatible, and fitting ideals commute with base change, $\{W_n\}$ are compatible: $W_n \times_{C_n} \Spec{C_m} \cong W_m$ for all $n \geq 1, m \leq n$. Hence the same is true for $\{Z_n\}$: $Z_n \times_{C_m} \Spec{C_m} \cong Z_m$, for all $n \geq 1, m \leq n$. 
Define $Q: = \varprojlim_n Q_n$, $\mathcal{O}_{Z}: = \varprojlim_n \mathcal{O}_{Z_n}$, and $\mathcal{O}_W :=  \varprojlim_n \mathcal{O}_{W_n}$ Then by the following equivalence of categories 
\cite[\href{https://stacks.math.columbia.edu/tag/087W}{Tag 087W}]{stacks-project}:
    \begin{equation} \label{equ: G_existence_thm_point}
    \begin{array}{rl}\Mod^{fp}(C) \rightarrow& \varprojlim_{n} \Mod^{fp}(C_n)\\
    M \mapsto& \{M \otimes_C C_n\}_{n \geq 1},
    \end{array}
    \end{equation}
we know that $Q$ is a finitely presented $C$-module, and both $\mathcal{O}_Z$ and $\mathcal{O}_W$ are finite $C$-algebras. Now since the inverse system $\{ \mathcal{O}_{W_n}\}$ satisfies the Mittag-Leffler condition, the exactness of \eqref{equ: ses_on_fitting_ideal} is preserved under taking inverse limits:
\begin{equation}
    0\to \mathcal{O}_W \to \mathcal{O}_Z \to Q \to 0.
\end{equation}
Notice that $Q \otimes_C C_n \cong Q_n$. Choose a basis $\bar{q}_1,...,\bar{q}_d$ of $Q_1$ as a $C_1$-vector space, and lift it to elements $q_1,..,q_d$ of $Q$. This defines a homomorphism $C^{\oplus d} \to Q$. 
After tensoring with $C_1$,
the induced map  $(C_n)^{\oplus d} \to Q_n$ becomes the isomorphism $(C_1)^{\oplus d} \to Q_1, e_i \mapsto \bar{q}_i$. Hence by Nakayama's lemma, each $(C_n)^{\oplus d} \to Q_n$ is a surjective map between modules of free of rank $d$ over $C_n$, and therefore  an isomorphism. It follows from the equivalence \eqref{equ: G_existence_thm_point} that $C^d \to Q$ is an isomorphism. So $Q$ is finite flat of rank $d$ over $C$.
Therefore, $Z \to W$ is a family of crimpings  over $C$. 
Now we claim the natural map $Z \hookrightarrow X_C$ is a closed immersion. Indeed, since $Z$ is finite over $C$ and $X_C$ is separated over $C$, the map $Z \to X_C$ is finite. Take an affine open $U: =\Spec R \subset X_C$, and set $V:= Z \times_{X_C} U$, which is again affine. After modulo $\mathfrak{m}$, the finite map $V \to U$ reduces to $V_1 = Z_1 \times_{X_{C_1}} U_1 \to U_1= \Spec R/\mathfrak{m} R$, which is an affine closed immersion. Therefore, by Nakayama's lemma, $V \to U$ is a closed immersion. This holds for every open affine $U \subset X_C$, and hence $Z \to X_C$ is a closed immersion. 
Define $Y$ as the pushout \[
\begin{tikzcd}
Z \ar[r] \ar[d, hook] & W \ar[d] \\ X_C \ar[r] & Y.
\end{tikzcd}
\]
Since $Z\to W$ is a family of crimpings over $C$,  and $Z \hookrightarrow X_C$ is a closed immersion, the map $f: X_C \to Y$ is a flat family of crimpings over $C$ by \Cref{prop: pushout}.

We are ready to define an inverse $\mathcal{F}$ to the restriction map $\mathcal{G}$: $\mathcal{F}$ assigns the family $f: X_C \to Y$ over $C$ constructed above to a compatible system  $\{(f_n: X_{C_n} \to Y_n)\}_{n\geq 1}$. From the construction, $f$ recovers $f_n$ after modulo $\mathfrak{m}^n$, and therefore, $\mathcal{G} \circ \mathcal{F} = \mathrm{id}$.  Conversely, starting with a flat family of crimpings $f: X_C \to Y$ over $C$, we define $W \hookrightarrow Y$ to be the closed subspace defined by $\Fitt_0 (Q)$, and define $Z:= W \times_Y X_C$. Restricting to $C_n$, we get a compatible system of flat families of crimpings $\{(f_n: X_{C_n} \to Y_n)\}_{n\geq 1}$.
Because fitting ideals commutes with base change, 
the two families of closed subspaces $\{Z_n \subset X_n\}$ and  $\{W_n \subset Y_n\}$ can be identified as $\{Z \times_C C_n\}$ and 
 and $\{ W \times_C C_n\}$. It follows from the equivalence $\eqref{equ: G_existence_thm_point}$ that the construction of $\mathcal{F}$ applied to $\{f_n\}$ recovers $Z$ and $W$. Finally as a pushout $Y = X_C \cup_Z W$, we recover the family $f: X_C \to Y$. 
 Therefore, $\mathcal{F} \circ \mathcal{G} = \mathrm{id}$. This shows bijectivity of $\mathcal{G}$.

\end{proof}

\begin{lem}[Strong homogeneity] 
\label{L: strong_homogeneity}
 Let $\Spec R \hookrightarrow \Spec R'$ be a nilpotent thickening of affine schemes over $S$, and let $\Spec R \to \Spec C$ be any morphism of affine schemes over $S$. Then the natural map 
\[
\mathcal{G}: \crimp^d_S (X)(R' \times_R C) \to 
\crimp^d_S (X) (R') \times_{\crimp^d_S (X) (R)} \crimp^d_S (X) (C)
\]
is a bijection. 
\end{lem}
\begin{proof}
We denote $C' := R' \times_R C$. An element in the right-hand side is the following data:
a flat family of crimpings $f_1: X_{R'} \to Y_1$  over $R'$, and a flat family of crimpings $f_2: X_C \to Y_2$  over $C$ so that after base change to $R$, $f_1$ and $f_2$ are identified as flat families of crimpings over $R$, which we call $f_0: X_R \to Y_0$. Define $W_i \subset Y_i$ to be the closed subspace defined by the fitting ideal sheaf $\Fitt_0(Q_i)$, and define $Z_0 := W_0 \times_{Y_0} X_R \subset X_R$, $Z_1 :=  W_1 \times_{Y_1} X_{R'} \subset X_{R'}$, and $ Z_2:= W_2 \times_{Y_2} X_C \subset X_C$. Now for every $i$, $f_i: Z_i \to W_i$ is a family of crimpings, and we have maps $Z_1 \hookleftarrow Z_0 \to Z_2$, and $W_1 \hookleftarrow W_0 \to W_2$, where $Z_0 \hookrightarrow Z_1$, $W_0 \hookrightarrow W_1$ are both nilpotent thickenings given that $\Spec R \hookrightarrow \Spec R'$ is a nilpotent thickening. We define $Z \to C'$ and $W \to C'$ as the pushout of the following diagrams respectively, 
\[
\begin{tikzcd}
Z_0 \ar[r, hook] \ar[d] & Z_1 \ar[d] \\ Z_2 \ar[r] & Z,
\end{tikzcd}
\quad \quad \quad \quad
\begin{tikzcd}
W_0 \ar[r, hook] \ar[d] & W_1 \ar[d] \\ W_2 \ar[r] & W.
\end{tikzcd}
\]
Since $Z_i$ and $W_i$ are finite over affine schemes, they are affine. Hence such pushouts exist and are computed as fiber products of rings.
Since $Z \times_{C'} \Spec C \cong Z_2$, and $\Spec C \to \Spec C'$ is a nilpotent thickening, by Nakayama's lemma, $Z$ is finite over $C'$. Similarly, $W$ is finite over $C'$.  It follows from the pushout construction that we have a natural map $g: Z \to W $, and  the following commutative diagram:
\[
\begin{tikzcd}
        0\arrow{r}{} & \mathcal{O}_W \arrow{r}{} \arrow[d,dashed]& 
        \mathcal{O}_{W_1} \oplus \mathcal{O}_{W_2}
        \arrow{r}{}\arrow{d} & \mathcal{O}_{W_0}\arrow{d} \arrow{r}{} & 0\\
        0 \arrow{r}{} & \mathcal{O}_Z \arrow{r}{}& \mathcal{O}_{Z_1} \oplus \mathcal{O}_{Z_2} \arrow{r}{} & \mathcal{O}_{Z_0} \arrow{r}{}  &  0.
\end{tikzcd}
\]
Since the vertical map in the middle is injective, the dashed map is also injective. Moreover, by the snake lemma, the cokernel $Q$ of $\mathcal{O}_W \hookrightarrow  \mathcal{O}_Z $ may be identified with the fiber product $Q_1 \times_{Q_0} Q_2$. Because $Q_1$ (resp. $Q_2$) is locally free of rank $d$ over $R'$ (resp. $C$), $Q$ is locally free of rank $d$ over $C'$. So $Z \to W$ is a family of crimpings over $C'$.

Note that the map $Z \to X_{C'}$ is a closed immersion. In fact, $Z \to X_{C'}$ is finite because $Z$ is finite, and $X_{C'}$ is separated over $C'$. Restricting to $C$, $Z_2 \to X_{C}$ is a closed immersion. Then Nakayama's lemma applied affine-locally shows that $Z \to X_{C'}$ is a closed immersion. 
So we may define  $\mathcal{F}: \crimp^d_S (X) (R') \times_{\crimp^d_S (X) (R)} \crimp^d_S (X) (C) \to  \crimp^d_S (X)(R' \times_R C)$, which sends a compatible  pair $f_1: X_{R'} \to Y_1, f_2: X_C \to Y_2$ to the  family $f: X_{C'} \to Y$ obtained from the pushout.  
\[
\begin{tikzcd}
Z \ar[r,"g"] \ar[d] & W \ar[d] \\ X_{C'} \ar[r, "f"] & Y.    
\end{tikzcd}
\]
By  \Cref{prop: pushout}, 
$f: X_{C'} \to Y$ is a flat family of crimpings over $C'$.
Then because $Y \times_{C'} \Spec R' \cong Y_1$, and $Y \times_{C'} \Spec C \cong Y_2$, we know that $\mathcal{G} \circ \mathcal{F} = \mathrm{id}$.

Suppose that $f: X_{C'} \to Y$ is a flat family of crimpings over $C'$. We define $W$ as the closed subspace of $Y$ cut out by  $\Fitt_0(Q)$, and $Z:= W \times_Y X_{C'}$. Now after base change to $R'$ and $C$, we get a pair of compatible flat families of crimpings $f_1: X_{R'} \to Y_1 \cong Y \times_{C'} \Spec R'$ and $f_2: X_{C} \to Y_2 \cong Y \times_{C'} \Spec C$.
Since fitting ideals commute with base change, $W_i \subset Y_i$ defined by $\Fitt_0(Q_i)$ coincides with the base change of $W$, and similarly  $Z_i$ is the base change of $Z$. Then the pushouts used in the construction of $\mathcal{F}$ recover $Z$ from $Z_i$, $W$ from $W_i$, and finally the original family $f: X_{C'} \to Y$. Thus $\mathcal{F} \circ \mathcal{G} =\mathrm{id}$. So the natural map $\mathcal{G}$ is a bijection.

\end{proof}

\begin{lem}[Diagonal]
\label{L: separatedness} Suppose that $S$ is a quasi-compact algebraic space, and $X \to S$ is of finite presentation.
    The diagonal morphism $\Delta: \crimp_S^d(X) \times \crimp_S^d(X)$ is a closed immersion. Therefore, $\crimp_S^d(X)$ is separated over $S$. 
\end{lem}
\begin{proof}
    Let $T$ be a scheme over $S$, with a map $T \to \crimp_S^d (X) \times \crimp_S^d (X)$ classifying two flat families of crimpings over $T$: $f_1: X_T \to Y_1$ and $f_2: X_T \to Y_2$. Consider the following Cartesian diagram, 
$$\begin{tikzcd}
\crimp_S^d(X) \times_{\crimp_S^d(X) \times \crimp_S^d(X)} T \arrow[r, "\Delta'"] \arrow[d] & T \arrow[d, "{(f_1 , f_2) }"] \\
\crimp_S^d(X) \arrow[r, "\Delta"] & \crimp_S^d(X) \times \crimp_S^d(X).
\end{tikzcd}$$ 
We want to show the morphism $\Delta'$ is a closed immersion into $T$.  The fiber product $ \mathcal{X}:= \crimp_S(X) \times_{\crimp_S(X) \times \crimp_S(X)} T$ is defined as a contravariant functor assigning  
$\underline{\mathrm{Isom}} (f_{1,T'}, f_{2,T'})$ to $( T' \to T)$.   Let $Z:= \Cond(f_1) \cup \Cond(f_2) \subset  X_T$, the union of two conductor subspaces of $f_1$ and $f_2$, and let $\mathcal{I}$ be the ideal sheaf defining $Z$. By construction, $f_{i,*} \mathcal{I}$ is contained in $\mathcal{O}_{Y_i}$, and we have a short exact sequence of $O_{Y_i}$-modules supported on $f_i(Z)$:
\[
0 \to \mathcal{O}_{Y_i}/(f_{i,*} \mathcal{I}) \to f_{i,\ast} (\mathcal{O}_{X_T}/\mathcal{I}) \to Q_i \to 0.
\]
Then because $Z$ and both $f_i(Z)$ are finite over $T$, the pushforward of two short exact sequences remain exact on $T$:
\[
0 \to \pi_{i,*}(\mathcal{O}_{Y_i}/(f_{i,*} \mathcal{I})) \to p_{2,\ast} (\mathcal{O}_{X_T}/\mathcal{I}) \to Q_i \to 0
\]
Here $\pi_i: Y_i \to T$ are the structure morphisms, and $p_2 : X \times_S T \to T$ is the second projection. Then $\Delta'(\mathcal{X})$ is defined by the vanishing locus of two homomorphisms of $\mathcal{O}_T$-modules: 
$\pi_{1,*}(\mathcal{O}_{Y_1}/(f_{1,*} \mathcal{I})) \hookrightarrow p_{2,\ast} (\mathcal{O}_{X_T}/\mathcal{I}) \twoheadrightarrow Q_2$, and $\pi_{2,*}(\mathcal{O}_{Y_2}/(f_{2,*} \mathcal{I})) \hookrightarrow p_{2,\ast} (\mathcal{O}_{X_T}/\mathcal{I}) \twoheadrightarrow Q_1$. Since both target $Q_i$ are finite locally free over $T$, the vanishing locus is a closed subscheme of $T$. Hence $\Delta'$ is a closed immersion.

\end{proof}

\begin{proof}[Proof of \Cref{T:main_algebraicity}]
Let $S$ be a quasi-compact, quasi-separated scheme. 
Using absolute Noetherian approximation\cite[\href{https://stacks.math.columbia.edu/tag/07SS}{Tag 07SS}]{stacks-project} we may write $S = \lim_i S_i$, where each $S_i$ is excellent, and by relative noetherian approximation $X \to S$ is the base change of some morphism $X' \to S_i$ of finite presentation. Furthermore, $\crimp^d_S(X) \cong S \times_{S_i} \crimp^d_{S_i}(X')$, so it suffices to prove algebraicity assuming $S$ is excellent. We verify  Artin's criteria for the \'etale stack $\crimp_S^d(X)$, as formulated in \cite[Theorem A] {MR3589351}.


Limit preservation, homogeneity, and effectivity (Condition (2),(3),(4)) are verified respectively 
in \Cref{lem: limit preservation}, \Cref{L: strong_homogeneity}, and \Cref{L: effectivity}. The coherence of deformation theory and obstruction theory (Condition (5),(6)) are verified in \Cref{P: coherence_of_def_obs}. This shows  $\crimp_S^d(X)$ is an algebraic space over $S$, locally of finite presentation. By \Cref{prop: pushout}, $\crimp_S^d(X)$ is quasi-compact. Hence $\crimp_S^d(X)$ is of finite presentation over $S$.  

By \Cref{L: separatedness},  $\crimp_S^d(X)$ is separated over $S$. 
If $X \to S$ is proper, $\cover_{d(d+1), d}$ is proper over $S$. Then admitting a proper surjective morphism from $\cover_{d(d+1), d}$ implies that  $\crimp_S^d(X)$ is proper over $S$. 

If $S$ is a quasi-compact algebraic space, we let $S' \to S$ be a quasi-compact, quasi-separated étale atlas. Then $\crimp_{S'}^d(X_{S'}) \cong  \crimp_S^d(X) \times_S S'$, 
and by previous discussion, $\crimp_{S'}^d(X_{S'})$ is algebraic, of locally finite presentation over $S'$.  Since algebraicity, locally presentation, and properness satisfy étale descent, and $\crimp_S^d(X)$ is separated and quasi-compact,  it follows that $\crimp_S^d(X)$ is represented by a separated  algebraic space of finite presentation over $S$, which is proper if $X \to S$ is proper. This completes the proof.    
\end{proof}

\subsection{The locus of bijective crimpings is closed}
\label{SS: geo_bij_locus}
Fix a base field $k$, and let $X \to S$ be of finite presentation. 
\begin{defn}
We say a crimping  $f: X_K \to Y$ over $K$ is \textit{geometrically bijective} if for any field extension $K \hookrightarrow L$, $f_L: X_L \to Y_L$ is a bijective map (of the underlying topological spaces).    
\end{defn} 
Since a crimping $f$ is finite and the induced map on structure sheaves is injective, it is integral and universally surjective. For a morphism of spaces, geometric injection is equivalent to universal injection. Hence a crimping being geometrically bijective is equivalent to being \emph{universally homeomorphic}. 

\begin{lem} 
\label{L: geo_bijec_closed}
Let $X \to S$ be of finite presentation. Then 
the locus in $\crimp_S^d(X)$ parametrizing geometrically bijective crimpings is closed.     
\end{lem}
\begin{proof}
    Let $f: X_T \to Y$ be a flat family of crimpings over $T$. We need to show the subset $F:= \{t \in T \mid f_{\bar{t}}= f \times \Spec \overline{k(t)}: X_t \times \Spec \overline{k(t)} \to Y_t  \times   \Spec \overline{k(t)} \text{ is bijective}\} \subset T$ is closed in $T$. It suffices to show that $F$ is constructible and stable under specialization. 
    Set $E:= X_T \times_Y X_T \setminus \Delta(X_T)$, where $\Delta: X_T \to X_T \times_Y X_T$ denotes the diagonal morphism. We note that $\pi: X_T \times_Y X_T \to T$ is of finite type, and moreover the condition that $f_{\bar{t}}$ is bijective is equivalent to the condition that $E_{\bar{t}}: = E_t \times \Spec \overline{k(t)} = \emptyset$. But $E_{\bar{t}} = \emptyset \iff E_{t} = \emptyset$, given that $ \Spec \overline{k(t)} \to \Spec k(t)$ is surjective. Therefore $T \setminus F = \pi(E)$. Because $f$ is finite and hence separated, $E$ is open and constructible. By Chevalley's theorem, $F$ is constructible. 

    To show that $F$ is stable under specialization, we may assume $T = \Spec R$, where $R$ is a DVR with fraction field $K$, and residue field $\kappa$. Assume $f_K$ is geometrically bijective. Suppose for contradiction that $f_\kappa$ is not geometrically bijective. Passing to an algebraic closure $\bar{\kappa}$ if necessary, we may assume $f_{\kappa}$ is not injective.   Let $y \in Y_{\kappa}$ be a point that has two distinct $x_1 \neq x_2 \in X_{\kappa}$ in the preimage. Replace $\mathcal{O}_{Y,y}$ by its strict henselization $B = \mathcal{O}_{Y,y}^{\mathrm{sh}}$, and let $A$ denote $(f_*\mathcal{O}_X) \otimes_{\mathcal{O}_Y} B$. Since $f_K$ is geometrically bijective, we may assume $\Spec A \to \Spec B$ is bijective over the generic fiber. 
    Because $f$ is finite, $A$ is a finite $B$-algebra, so it can be written as a product of local rings. Since $\Spec A \to \Spec B$ has at least two different points lying over the closed point of $\Spec B$, it follows that $A$ cannot be local and therefore the product is nontrivial. Let $e\in A$ be a nontrivial idempotent. Since $A$ is flat over $R$, $e_K \in A_K$ remains nontrivial. 
    Over the generic fiber, $\Spec A_K \to \Spec B_K$ is bijective, and therefore the nontrivial  idempotent $e_K \in B_K$. Then $e \in B$. Indeed, let $\pi$ denote the uniformizer of $R$. If $e \notin B$, there exists a minimal integer $n \geq 1$ such that $\pi^n e \in B$. But this implies $\pi^{n-1} e$ is a $\pi$-torsion element of $Q= A/B$. This contradicts the $R$-flatness of $Q$. By assumption, $B$ is local, and hence has no nontrivial idempotent. We obtain a contradiction. Therefore, $f_{\kappa}$ is geometrically bijecitve, and $F$ is stable under specialization. Together with the constructibility of $F$, we conclude that $F$ is closed.
    
\end{proof}
On the geometrically bijective locus, a (flat) family of crimpings may be regarded as a (flat) family of subalgebras. 
\begin{lem}
Given a (flat) family of geometrically bijective crimpings $f: X \to Y$ of corank $d$ over $T$. Let $U =\Spec A \subset X$. Then the image $f(U) = \Spec B$ is an affine open in $Y$, given by a  subalgebra $B$  of $A$.      
\end{lem}
\begin{proof}
By hypothesis, $f$ is a bijection between the underlying topological spaces. Then, $\mathcal{O}_Y$ may be regarded as a sheaf on $X$, and in particular, $\mathcal{O}_Y \subset \mathcal{O}_X$ is a sheaf of subalgebras. Suppose that $U =\Spec A$ is an affine open of $X$. Since $f$ is a homeomorphism and therefore an open map, it follows that $f(U)$ is open in $Y$. Because $f|_U : U \to V$ is finite and surjective, and affineness descends under finite surjective morphism, $V$ is affine.  
\end{proof}

\section{Semistability for curves of geometric genus $0$}
\label{SS: stability}

\subsection{Universal homeomorphisms and intrinsic moduli theory}

Here we observe a general statement that we will subsequently use to study the moduli of genus $0$ curves.

\begin{prop} \label{P: general_statement}
    Suppose that $\pi: \mathcal{X} \to \mathcal{Y}$ is a representable universal homeomorphism between algebraic stacks locally of finite presentation, quasi-separated, and with affine automorphism groups and separated diagonal, all relative to a noetherian base stack. Then
    \begin{enumerate}
    
    \item \label{P: gs_1} For every $n \geq 1$,  $\Grad^n (\mathcal{X}) \to \Grad^n (\mathcal{Y})$ is a representable universal homeomorphism. 
        \item \label{P: gs_2} Any numerical invariant on $\mathcal{X}$ is pulled back from a unique numerical invariant on $\mathcal{Y}$. One numerical invariant determines a weak $\Theta$-stratification if and only if the other does, and the stratification of $\mathcal{X}$ is induced from the stratification of $\mathcal{Y}$ in this case.
        \item \label{P: gs_3} If the base has characteristic $0$, then $\mathcal{X}$ admits a separated good moduli space if and only if $\mathcal{Y}$ does.
    \end{enumerate}
\end{prop}
\begin{proof}[Proof of \ref{P: gs_1}:]
    Let us denote $\bG_m^n$ by $G$. 
    Let $T$ be a scheme, and let $T \to \Grad(\mathcal{Y})$ correspond to a map $\lambda: (BG)_T \to \mathcal{Y}$. We show that the fiber product $T \times_{\Grad(\mathcal{Y})} \Grad(\mathcal{X})$ is an algebraic space, and the projection 
    \[
    T \times_{\Grad(\mathcal{Y})} \Grad(\mathcal{X}) \to T
    \]
    is a universal homeomorphism. Pulling back $\lambda$ along the atlas $T \to (BG)_T$, we get a $T$-point $y: T \to \mathcal{Y}$.  Set $Z:= T \times_{\mathcal{Y}} \mathcal{X}$. Under the hypotheses on the morphism $\mathcal{X} \to \mathcal{Y}$ , $Z$ is an algebraic space, and $Z \to T$ is a quasi-separated locally finitely presented universal homeomorphism. Moreover, by descent, $Z$ is naturally equipped with a $G$-action over $T$. By \cite[Proposition 1.4.1]{halpernleistner2022structureinstabilitymodulitheory}, the fixed locus $Z^{G}$ is represented by an algebraic space, and the natural morphism $Z^{G} \hookrightarrow Z$ is a closed immersion. Take any geometric point $z: \Spec k \to Z$, where $k$ is an algebraically closed field. Composing with $Z \to T$ gives a geometric point $t: \Spec k \to T$, and the map $z$ factors through the fiber $Z_t:= Z \times_T \Spec k$. Since the action of $G$ on $Z$ is over $T$, $Z_t$ is $G_k$-invariant. Because $Z \to T$ is a universal homeomorphism, $Z_t$ contains exactly one underlying point, and $(Z_t)_{\rm{red}} \cong \Spec k$. The orbit map $o_z: G_k \to Z_t$ of $z$ under $G_k$-action factors through $(Z_t)_{\rm{red}}$,  because $G_k$ is reduced. So it has to be constant and $z$ is fixed under $G$-action. It follows that every geometric point lies in the fixed locus.  Therefore the closed immersion $Z^{G} \to Z$ is surjective  and hence a universal homeomorphism. 

    We note that $T \times_{\Grad(\mathcal{Y})} \Grad(\mathcal{X})$ is canonically identified with $Z^{G}$. Indeed, a $T'$-point of $T \times_{\Grad(\mathcal{Y})} \Grad(\mathcal{X})$ corresponds to a map $(BG)_{T'} \to X$ lifting  $(BG)_{T'} \to (BG)_T \xrightarrow{\lambda}\mathcal{Y}$, which is equivalent to a section of $(BG)_{T'} \times_{\mathcal{Y}} \mathcal{X} \to (BG)_{T'}$. After pulling back along the atlas $T' \to (BG)_{T'}$, such a  section becomes a $G$-equivariant section of $Z_{T'} \to T'$,  and this is the same as a $T'$-point of $Z^{G}$. Since we have shown that $Z^{G}$ is represented by an algebraic space and  $Z^{G} \to T$ is a universal homeomorphism, this completes the proof. 

    \medskip
    \noindent\textit{Proof of \ref{P: gs_2}:}
    \medskip
    
    Let $\mu^{\mathcal{X}}$ be a numerical invariant on $\mathcal{X}$. 
    Recall that a numerical invariant with values in $\Gamma$ assigns a scaling-invariant function of $\bR^n \setminus \{0\} $ to $\Gamma$ to every nondegenerate point of $|\Grad^n(X)|$. The assignment is locally constant, unchanged under field extensions, and compatible with restrictions along homomorphisms $\bG_m^q \to \bG_m^n$ with finite kernel \cite[Definition 0.0.3]{halpernleistner2022structureinstabilitymodulitheory}. 
    Let $p \in \mathcal{Y}(k)$, and let $\gamma: (\bG_m^n)_k \to \Aut_{\mathcal{X}}(p)$ be any homomorphism of $k$-groups with finite kernel. Together the pair $(p, \gamma)$ defines a $k$-point of $\Grad^n(\mathcal{X})$. Passing to an algebraic closure $\bar{k}/k$, since $\Spec \bar{k} \times_{\Grad^n(\mathcal{Y})} \Grad^n(\mathcal{X}) \to \Spec \bar{k}$ is a universal homeomorphism, $\Spec \bar{k} \times_{\Grad^n(\mathcal{Y})} \Grad^n(\mathcal{X})$ has a unique geometric point. Thus $\Spec \bar{k} \to \Spec \bar{k} \times_{\Grad^n(\mathcal{Y})} \Grad^n(\mathcal{X}) \to \Grad^n(\mathcal{X})$ gives a unique lift of the pair $(p, \gamma)$. Write this lift as $(\tilde{p}, \tilde{\gamma})$, where $\tilde{p} \in \mathcal{X}(\bar{k})$, and $\tilde{\gamma}: (\bG^n_m)_{\bar{k}} \to \Aut_{\mathcal{X}}(\tilde{p})$ a homomorphism of $\bar{k}$-groups, satisfying that $\pi(\tilde{p}) = p$. Since $\pi: \mathcal{X} \to \mathcal{Y}$ is representable, $\Aut_{\mathcal{X}}(\tilde{p}) \to \Aut_{\mathcal{Y}}(p)$ is a monomorphism. Thus $\ker (\tilde{\gamma}) = \ker (\gamma)$ is finite. Define a scaling-invariant function   $\mu^{\mathcal{Y}}_{\tilde{\gamma}}: \mathbb{R}^n \setminus \{0\} \to \Gamma$
    by $\mu^{\mathcal{Y}}_{\tilde{\gamma}} (v) := \mu_{\gamma}^{\mathcal{X}}(v)$, for any $ 0\neq v \in \mathbb{R}^n$.  If $\bar{k}'$ is another algebraic closure of $k$, one may embed both $\bar{k}$ and $\bar{k}'$ into a common algebraically closed extension. Because $\mu^{\mathcal{X}}$ is a numerical invariant and hence the assignment of scaling-invariant functions is unchanged under field extension, the definition of $\mu^{\mathcal{Y}}_{\tilde{\gamma}}$ is independent of the choice of an algebraic closure. 
    It is not difficult to show that the functions  $\mu^{\mathcal{Y}}_{\tilde{\gamma}}$ define a numerical invariant $\mu^{\mathcal{Y}}$ on $\mathcal{Y}$, and we leave details to the reader.

    Note that $\pi$ is integral and locally of finite presentation, and hence it is a finite morphism. By \cite[Proposition 1.3.2]{halpernleistner2022structureinstabilitymodulitheory}, the morphism $\Filt(\mathcal{X}) \to \Filt(\mathcal{Y})$ is finite and representable, and the canonical morphism $\Filt(\mathcal{X}) \to \Filt(\mathcal{Y}) \times_{\mathcal{Y}} \mathcal{X}$ is a surjecitve closed immersion. Since a surjecitve closed immersion is a universal homeomorphism, and the base change $\Filt(\mathcal{Y}) \times_{\mathcal{Y}} \mathcal{X} \to \Filt(\mathcal{Y})$ is a universal homeomorphism as well, it follows that $\Filt(\pi): \Filt(\mathcal{X}) \to \Filt(\mathcal{Y})$ is a finite universal homeomorphism. Moreover, the following square commutes:
    \[
    \begin{tikzcd}
     \Filt(\mathcal{X}) \ar[r, "\Filt(\pi)"] \ar[d, "\mathrm{ev}_1^{\mathcal{X}}"] &  \Filt(\mathcal{Y}) \ar[d, "\mathrm{ev}_1^{\mathcal{Y}}"]  \\
     \mathcal{X} \ar[r,"\pi"] & \mathcal{Y}. 
    \end{tikzcd}
    \]
    
    Hence if $q \in |\mathcal{X}|$ and $p = \pi(q)$, $\Filt(\pi)$ identifies the sets of filtrations of $q$ and $p$. Therefore, for stability functions attached to $\mu^\mathcal{X}$ and $\mu^\mathcal{Y}$, we have $M^{\mu^\mathcal{X}}(q) = M^{\mu^\mathcal{Y}}(p).$ One direction is clear: if $\mu^{\mathcal{Y}}$ defines a weak $\Theta$-stratification on $\mathcal{Y}$, because $\pi$ is finite, \cite[Lemma 2.3.2, Lemma 4.1.19]{halpernleistner2022structureinstabilitymodulitheory} implies the induced weak $\Theta$-stratification on $\mathcal{X}$ is determined by $\mu^\mathcal{X}$.

    Now suppose that $\mu^\mathcal{X}$ defines a weak $\Theta$-stratification on $\mathcal{X}$ corresponding to open substacks $\{\mathcal{X}_{\leq c}\}_{c \in \Gamma}$ and strata $\{\mathcal{S}_c^{\mathcal{X}} \subset \Filt(\mathcal{X}_{\leq c})\}_{c \in \Gamma}$. Define open substacks $\mathcal{Y}_{\leq c} \subset \mathcal{Y}$ whose underlying topological space is given by $|\mathcal{Y}_{\leq c}| = \pi(|\mathcal{X}_{\leq c}|)$, which coincides with the open subset defined numerically as $\{y \in |\mathcal{Y}| \mid M^{\mu^\mathcal{Y}}(y) \leq c\}$. Let $|\mathcal{S}_c^{\mathcal{Y}}| := \Filt(\pi) (|\mathcal{S}_c^{\mathcal{X}}|)$, which is open and closed in $|\Filt(\mathcal{Y}_{\leq c})|$, and thus  a union of connected components of $\Filt(\mathcal{Y}_{\leq c})$. We let $\mathcal{S}_c^{\mathcal{Y}} \subset \Filt(\mathcal{Y}_{\leq c})$ be the corresponding open and closed substack.  We have a commutative diagram in which the square is Cartesian, and the map $\mathcal{S}_c^{\mathcal{X}} \to \mathcal{S}_c^{\mathcal{Y}}
\times_{\mathcal{Y}_{\leq c}}
\mathcal{X}_{\leq c}$ is a surjective closed immersion,
    \[
    \begin{tikzcd}
     \mathcal{S}_c^{\mathcal{X}}
\ar[r]
\ar[rd, bend right=25, "\ev_1^{\mathcal{X}}"]
\ar[rr, bend left=25, "\Filt(\pi)"]
&
\mathcal{S}_c^{\mathcal{Y}}
\times_{\mathcal{Y}_{\leq c}}
\mathcal{X}_{\leq c}
\ar[r] \ar[d]
&
\mathcal{S}_c^{\mathcal{Y}}
\ar[d, "\ev_1^{\mathcal{Y}}"]
\\
     & \mathcal{X}_{\leq c} \ar[r,"\pi"]  & \mathcal{Y}_{\leq c}. 
    \end{tikzcd}
    \]
    It remains to show that every $\mathcal{S}_c^{\mathcal{Y}}$ is a weak $\Theta$-stratum of $\mathcal{Y}_{\leq c}$, equivalently, 
    $\ev_1^{\mathcal{Y}}: \mathcal{S}_c^{\mathcal{Y}} \to \mathcal{Y}_{\leq c}$ is  finite and radicial. 
     To prove that it is representable, it suffices to show the relative inertia $I_{\mathcal{S}_c^{\mathcal{Y}}/ \mathcal{Y}_{\leq c}} \to \mathcal{S}_c^{\mathcal{Y}}$ is trivial. Let $f$ be a geometric point of $\mathcal{S}_c^{\mathcal{Y}}$ with $f(1) = p$, and let $g$ be the unique lift of $f$ in $\mathcal{S}_c^{\mathcal{X}}$ with $g(1) = q \in |\mathcal{X}_{\leq c}|$. By definition, $(I_{\mathcal{S}_c^{\mathcal{Y}}})_{f} = \ker (\Aut_{\mathcal{S}_{c}^{\mathcal{Y}}} (f) \to \Aut_{\mathcal{Y}}(p))$. Because $\mathcal{S}_c^{\mathcal{X}} \to \mathcal{S}_c^{\mathcal{Y}}
\times_{\mathcal{Y}_{\leq c}}
\mathcal{X}_{\leq c}$ is a surjective closed immersion, we have $\Aut_{\mathcal{S}_c^{\mathcal{X}}} (g) \cong \Aut_{\mathcal{S}_c^{\mathcal{Y}}}(f) \times_{\Aut_{\mathcal{Y}}(p)} \Aut_{\mathcal{X}}(q)$. So $\ker (\Aut_{\mathcal{S}_{c}^{\mathcal{Y}}} (f) \to \Aut_{\mathcal{Y}}(p)) \cong \ker (\Aut_{\mathcal{S}_c^{\mathcal{X}}} (g) \to \Aut_{\mathcal{X}} (q))$. Since $\ev_1^{\mathcal{X}}: \mathcal{S}_c^{\mathcal{X}} \to \mathcal{X}_{\leq c}$ is representable, $\ker (\Aut_{\mathcal{S}_c^{\mathcal{X}}} (g) \to \Aut_{\mathcal{X}} (q)) = 1$ and thus $(I_{\mathcal{S}_c^{\mathcal{Y}}})_{f} = 1$ as a trivial group scheme. The relative inertia $I_{\mathcal{S}_c^{\mathcal{Y}}/ \mathcal{Y}_{\leq c}} \to \mathcal{S}_c^{\mathcal{Y}}$ is locally of finite type and has trivial fibers. Therefore it is unramified. Hence the identity section $e: \mathcal{S}_c^{\mathcal{Y}} \to I_{\mathcal{S}_c^{\mathcal{Y}}/ \mathcal{Y}_{\leq c}}$ is an open immersion and surjective on geometric points, and thus an isomorphism. This makes the relative inertia trivial over $B$. 

    Since the composition $\mathcal{S}_c^{\mathcal{X}} \xrightarrow{\Filt(\pi)} \mathcal{S}_c^{\mathcal{Y}} \xrightarrow{\ev_1^{\mathcal{Y}}} \mathcal{Y}_{\leq c}$ is finite, and $\Filt(\pi)$ is a finite universal homeomorphism and thus a finite surjective morphism, it follows that $\ev_1^{\mathcal{Y}}: \mathcal{S}_c^{\mathcal{Y}} \to \mathcal{Y}_{\leq c}$ is finite. By lifting geometric points of $\mathcal{S}_c^{\mathcal{Y}}$
    uniquely along $\Filt(\pi)$ to $\mathcal{S}_c^{\mathcal{X}}$, injectivity of the composition $\mathcal{S}_c^{\mathcal{X}} \to \mathcal{X}_{\leq c} \to \mathcal{Y}_{\leq c}$ implies injectivity of $\mathcal{S}_c^{\mathcal{Y}} \to \mathcal{Y}_{\leq c}$. This argument remains true after any base change $T \to \mathcal{Y}_{\leq c}$ because universal homeomorphisms and radicial morphisms are preserved under base change. Hence $\mathcal{S}_c^{\mathcal{Y}} \to \mathcal{Y}_{\leq c}$ is universal injective, which implies that it is radicial because the map is representable and finite. 

    \medskip
    \noindent\textit{Proof of \ref{P: gs_3}:}
    \medskip
    
    Let the base have characteristic $0$. Since $\pi: \mathcal{X} \to \mathcal{Y}$ is integral and thus affine, \cite[Lemma 4.14]{AIF_2013__63_6_2349_0} implies that if $\mathcal{Y}$ admits a separated good moduli space, so does $\mathcal{X}.$  
    
    Assume that  $\mathcal{X}$ admits a separated good moduli space $q_\mathcal{X}: \mathcal{X} \to X$. Suppose first that $\mathcal{X}$ is quasi-compact. Then 
    \cite[Theorem A]{Alper_2023} implies that $X$ is $\Theta$-reductive and $S$-complete. Let $R$ be any DVR with a uniformizer $u \in R$. Let $T$ denote either $\Spec R[t]$ or $\Spec R[s,t]/ (st -u)$ with the standard $\bG_m$ action \footnote{For $T = \Spec R[t]$, $\bG_m$ acts with weight $-1$ on $t$; for $T = \Spec R[s,t]/ (st -u)$, $\bG_m$ acts with weight $-1$ on $t$, and weight $1$ on $s$.}. Let $0 \in T$ be the unique closed point fixed by $\bG_m$. Since the base is of characteristic $0$, $T \setminus 0$ is regular and therefore absolutely weakly normal \cite[Appendix B]{BSMF_2010__138_2_181_0}. Recall that the absolute weak normalization of an algebraic space $Z$ is denoted by $Z^{\mathrm{wn}}$, and it is initial in the category of separated universal homeomorphisms $Z' \to Z$. Given a morphism $ (T \setminus 0)/ \bG_m \to \mathcal{Y}$, after base change to $T \setminus 0$, the morphism 
    \[
    (T\setminus 0) \times_{\mathcal{Y}} \mathcal{X} \to  T\setminus 0
    \]
    is again a separated universal homeomorphism. Because $(T\setminus 0)^{\mathrm{wn}} = (T\setminus 0)$, this morphism admits a section that descents to give a unique lift of $(T \setminus 0)/ \bG_m \to \mathcal{Y}$ to $\mathcal{X}$. Using this unique lift, it follows that $\mathcal{Y}$ is $\Theta$-reductive and $S$-complete. Therefore, by  \cite[Theorem A]{Alper_2023}, $Y$ admits a separated good moduli space. 

    Now if $\mathcal{X}$ is not quasi-compact, choose a quasi-compact open cover $\cup_i X_i = X$, and set $\mathcal{X}_i := q_\mathcal{X}^{-1}(X_i)$, and $\mathcal{Y}_i := \pi(\mathcal{X}_i)$. Then each $\mathcal{Y}_i$ is a quasi-compact open substack of $\mathcal{Y}$, and $\cup_i \mathcal{Y}_i = \mathcal{Y}$. By the previous case, each $\mathcal{Y}_i$ admits a separated good moduli space $q_i: \mathcal{Y}_i \to Y_i$. Write $\mathcal{X}_{ij} = \mathcal{X}_i \cap \mathcal{X}_j$, and $\mathcal{Y}_{ij} = \mathcal{Y}_i \cap \mathcal{Y}_j$. Then by \cite[Proposition 7.9]{AIF_2013__63_6_2349_0}, these good moduli space maps glue to give a good moduli space map $q_{\mathcal{Y}}: \mathcal{Y} \to Y$ if and only if  $\mathcal{Y}_{ij} \to  \mathcal{Y}_i$ is saturated for each $i, j$. \cite[Lemma 3.32]{Alper_2023} implies that it suffices to show the inclusion $\mathcal{Y}_{ij} \to  \mathcal{Y}_i$ is $\Theta$-surjective, i.e. the natural morphism $\Filt(\mathcal{Y}_{ij}) \to \mathcal{Y}_{ij} \times_{\mathcal{Y}_{i}, \ev_1} \Filt(\mathcal{Y}_i)$ is surjective. Let $f: \Theta_k \to \mathcal{Y}_i$ be a filtration with $f(1) \in \mathcal{Y}_{ij}$. Since $\Filt(\mathcal{X}_i) \to \Filt(\mathcal{Y}_i)$ is a finite universal homeomorphism, there exists a unique filtration $\tilde{f}: \Theta_k \to \mathcal{X}_i$ such that $\pi \circ \tilde{f} = f$. Hence in particular, $\tilde{f}(1) \in \mathcal{X}_{ij}$. Then the map $\Theta \xrightarrow{\tilde{f}} \mathcal{X}_i \to X_i$ factors through the good moduli space map $\Theta \to \Spec k$. Therefore the composition is constant and $\tilde{f}(\Theta) \subset q_\mathcal{X}^{-1}(X_{ij}) = \mathcal{X}_{ij}$.  So $f(\Theta_k) \subset \mathcal{Y}_{ij}$ and it is the image of some filtration of $\mathcal{Y}_{ij}$.  Moreover, because $\mathcal{X}$ is $S$-complete \cite[Proposition 3.48 (2)]{Alper_2023}, it follows that $\mathcal{Y}$ is also $S$-complete. Thus again by \cite[Proposition 3.48 (2)]{Alper_2023}, the good moduli space $Y$ is separated over the base. 
    \end{proof}

\subsection{The stack of geometrically unibranch genus $0$ curves}
Let $\Bbbk$ be a base field of characteristic $0$. Let $k/\Bbbk$ be a field extension. 

\begin{defn}

A crimping $f: \bP^1_k \to C$ is called a geometrically unibranch crimping if the target curve $C$ is geometrically unibranch (in the sense of \cite[\href{https://stacks.math.columbia.edu/tag/0BQ2}{Tag 0BQ2}]{stacks-project}). From the definition, a crimping of $\bP^1_\Bbbk$ is geometrically bijective if and only if it is  geometrically unibranch. 
    We let $\crimp^d_{\Bbbk} (\bP^1_{\Bbbk})^{\rm{un}} \subset \crimp^d_{\Bbbk} (\bP^1_{\Bbbk})$ denote the reduced closed subspace of points that correspond to geometrically bijective crimpings (see \Cref{L: geo_bijec_closed}). \footnote{$\SL_2$ acts on $\crimp^d_{\Bbbk} (\bP^1_{\Bbbk})$ by precomposing the action on $\bP^1_\Bbbk$, and it sends geometrically unibranch crimpings to geometrically unibranch crimpings. Therefore,   $\crimp^d_{\Bbbk} (\bP^1_{\Bbbk})^{\rm{un}}$ is $\SL_2$-invariant.}We define $\cU_{0,d} := \crimp^d_{\Bbbk} (\bP^1_{\Bbbk})^{\rm{un}}/\SL_2$ to be the stack of geometrically unibranch genus $0$ curves.
\end{defn}

Noting that for $\bP^1$, there is a proper surjective morphism $\cover_{2d,d}(\bP^1_\Bbbk) \to \crimp^d_\Bbbk (\bP^1_\Bbbk)$ by \Cref{R: curve_bound}, compatible with $\SL_2$-action, we set $$\cover_{2d,d}^{\rm{un}} (\bP^1_\Bbbk):= \cover_{2d,d}(\bP^1_\Bbbk) \times_{\crimp^d_\Bbbk (\bP^1_\Bbbk)} \crimp^d_{\Bbbk} (\bP^1_{\Bbbk})^{\rm{un}}.$$ 
Then $\cover_{2d,d}^{\rm{un}} (\bP^1_\Bbbk) \hookrightarrow \cover_{2d,d}(\bP^1_\Bbbk)$ is an $\SL_2$-invariant closed subscheme, and $\cover_{2d,d}^{\rm{un}} (\bP^1_{\Bbbk}) \to \crimp^d_{\Bbbk} (\bP^1_{\Bbbk})^{\rm{un}}$ is again proper surjective. Let $
    \rho: \cover_{2d,d}^{\rm{un}} (\bP^1_{\Bbbk}) \to \crimp^d_{\Bbbk} (\bP^1_{\Bbbk})^{\rm{un}}$ denote the $\SL_2$-equivariant proper surjective morphism.

Let $f: \bP^1_k \to C$ be a geometrically unibranch crimping of corank $d$. For any closed point $p \in \bP^1_k$, let $\delta_p (f):= \len_{\cO_{C, f(p)}} (Q_{f(p)})$. Define 
\[
Z_f: = \sum_{p \in |\bP^1_k|} 2\delta_p (f) \cdot [p] \subset \bP^1_k. 
\]
This is a closed subscheme of points of length $2d$, because $\sum_{p \in |\bP^1_k|} 2 \len_{\cO_{C, f(p)}} (Q_{f(p)}) \cdot [k(f(p)): k] = \sum_{p \in |\bP^1_k|} 2 \dim_k Q_{f(p)} = 2d$. Noting that the morphism $f$ is radicial, the extension $k(f(p)) \subset k(p)$ is purely inseparable. Because we are working with fields of characteristic $0$, such extension is trivial, and therefore, $k(f(p)) = k(p)$.  
Let $\mathcal{I}_{Z_f}$ denote its ideal sheaf. Then by the local conductor bound discussed in \Cref{R: curve_bound}, we have $\mathcal{I}_{Z_f} \subset \Ann_C (Q) \subset \mathcal{O}_C \subset  \cO_{\bP^1_k}$. Therefore $\mathcal{O}_{W_f}:=\mathcal{O}_C/ \mathcal{I}_{Z_f}$ is a subalgrbra of $\mathcal{O}_{Z_f}$, and $\mathcal{O}_{Z_f}/\mathcal{O}_{W_f} = Q$ is free of rank $d$. Hence $[W_f \leftarrow Z_f \hookrightarrow \bP^1_k]$ gives a $k$-point of $\cover_{2d,d}(\bP^1_{\Bbbk})$.

Since the image of $[W_f \leftarrow Z_f \hookrightarrow \bP^1_k]$ under the morphism $\cover_{2d,d}(\bP^1_\Bbbk) \to \crimp^d_\Bbbk (\bP^1_\Bbbk)$ recovers the geometrically unibranch crimping $f$,  $[W_f \leftarrow Z_f \hookrightarrow \bP^1_k]$ is in fact a $k$-point of $\cover_{2d,d}^{\rm{un}}(\bP^1_\Bbbk)$. 
We obtain a set-theoretic lift $\psi$:  for every field $k$ over $\Bbbk$, 
\begin{equation}
    \psi_k: \crimp^d_{\Bbbk} (\bP^1_{\Bbbk})^{\rm{un}}(k) \to \cover_{2d,d}^{\rm{un}}(\bP^1_{\Bbbk})(k), \quad (f: \bP^1_k \to C) \mapsto  [W_f \leftarrow Z_f \hookrightarrow \bP^1_k]. 
\end{equation}

Let $L/k$ be a field extension. For the base change crimping $f_L: \bP^1_L \to C_L$, and a closed point $q \in \bP^1_L$ lying over $p \in \bP^1_k$, $(Q_L)_{f_L(q)} \cong Q_{f(p)} \otimes_{\cO_{C, f(p)} } \cO_{C_L, f_L(q)}$. Therefore, $\delta_q (f_L) = \len_{\cO_{C_L, f_L(q)}}((Q_L)_{f_L(q)}) = \len_{\cO_{C, f(p)}} (Q_{f(p)}) = \delta_p(f)$. So $Z_{f_L} = Z_f \times_k \Spec L $, and thus $\psi_L(f_L) = \psi_k(f) \times_{k} \Spec L$. Therefore $\psi$ is a well-defined map of the topological space $|\crimp^d_{\Bbbk} (\bP^1_{\Bbbk})^{\rm{un}}|$ to $|\cover_{2d,d}^{\rm{un}}(\bP^1_{\Bbbk})|$. 
It is clear that $\psi$ is $\SL_2$-equivariant on all field-valued points. 


 \begin{lem}\label{L: lift_DVR}
     Let $R$ be a DVR over $\Bbbk$, with fraction field $K$, and residue field $\kappa$. Let $f: \bP^1_R \to C$ be a flat $R$-family of geometrically unibranch crimpings of corank $d$. The $K$-point $\psi_K(f_K) \in \cover_{2d,d} (\bP^1_{\Bbbk})(K)$ extends uniquely to an $R$-point $\xi = \xi(f)$ of $\cover_{2d,d}^{\rm{un}} (\bP^1_{\Bbbk})$. Then its special fiber is identified with $\psi_\kappa(f_\kappa)  $. Moreover, this construction is $\SL_2$-equivariant: for every $g \in \SL_2(R)$, $\xi (g \cdot f) = g \cdot \xi (f)$. 
 \end{lem}
 \begin{proof} Suppose that this $R$-family is given by a sheaf of subalgebras $\mathcal{B} \subset \cO_{\bP^1_R}$.
 Let $Z_R \subset \bP^1_R$ be the scheme-theoretic closure of $Z_{f_K}$. Then $Z_R$ is a finite flat closed subscheme of length $2d$, and $(Z_R)_K = Z_{f_K}$. Moreover, we claim that $(Z_R)_{\kappa} = Z_{f_{\kappa}}$. Let $x \in \bP^1_{\kappa}$. We show the equality by comparing the coefficients of $[x]$ on both sides. Write $(Z_R)_{\kappa} = \sum_{x \in |\bP^1_{\kappa}|} a_x [x]$, where $a_x = \len_{\mathcal{O}_{\bP^1_\kappa,x}} \cO_{(Z_R)_\kappa, x}$. Then the contribution to length is preserved under specialization, $$a_x [k(x): \kappa] = \sum_{p \in |\bP^1_K|, p \to x} 2 \delta_p(f_K) [k(p): K],$$ summing over all points $p$ in the generic fiber specializing to $x$. On the other hand, 
 since $f: \bP^1_R \to C$ is a homeomorphism, we may regard the quotient sheaf $Q$ as a sheaf on $\bP^1_R$, though it might not be an $\cO_{\bP^1_R}$-module.   
 We choose an open neighborhood $U \subset \bP^1_R$ of $x$ such that $U_\kappa$ contains no other limit point of the support of $Q_K$. Then the restriction of $Q$ to $\supp(Q) \cap U$ is a direct summand of $Q$, and therefore is again flat over $R$. It follows from $R$-flatness that $$\sum_{p \in |\bP^1_K|, p \to x}\dim_K (Q_K)_{f_K(p)} = \dim_\kappa (Q_\kappa)_{f_\kappa (x)}.$$
 Equivalently, 
 $$\sum_{p \in |\bP^1_K|, p \to x} \delta_p(f_K) [k(f_K(p)): K] = \delta_x(f_\kappa) [k(f_\kappa(x)): \kappa]. $$ Since $f$ is radicial, it follows that $k(f_K(p)) = k(p)$, $k(f_\kappa(x)) = k(x)$.  Therefore, $2 \delta_x(f_\kappa) = a_x$ for every $x \in \bP^1_{\kappa}$. This shows the claim. 

 Now let $\mathcal{I}_{Z_R}$ be the ideal sheaf of $Z_R$. The composition $\mathcal{I}_{Z_R} \to \cO_{\bP^1_R}  \to Q$ vanishes over the generic fiber, and hence its image is $R$-torsion. Since the image is inside a finite flat module $Q$ over $R$, it follows that the map vanishes on $ \bP^1_R$. Therefore, $\mathcal{I}_{Z_R} \subset \mathcal{B}$. We define $W_R$ to be the finite scheme over $R$ with structure sheaf $\mathcal{B} / \mathcal{I}_{Z_R}$. Then $\xi = \xi(f) = [W_R \leftarrow Z_R \hookrightarrow \bP^1_R]$ is an extension of $\psi_K(f_K)$. Moreover, $(W_R)_\kappa$ is given by $(\mathcal{B})_\kappa / (\mathcal{I}_{Z_R})_\kappa = (\mathcal{B})_\kappa/ \mathcal{I}_{Z_{f_\kappa}}$ by the claim. Hence $\xi_\kappa = \psi_\kappa(f_\kappa)$.

Let $g \in \SL_2(R)$. Then the translated family $g \cdot f$ corresponds to the subalgebra 
$g \cdot \mathcal{B} \subset \cO_{\bP^1_R}$. Since $\psi$ is pointwise $\SL_2$-equivariant, $g \cdot Z_{f_K} = Z_{(g\cdot f)_K}$. Hence for the scheme-theoretic closure, $ Z_R (g\cdot f)=  \overline{Z_{(g\cdot f)_K}} = \overline{g \cdot Z_{f_K}} = g \cdot Z_R(f)$. Equivalently, we have $\mathcal{I}_{Z_R (g\cdot f)} = g \cdot \mathcal{I}_{Z_R(f)}$. Therefore $\mathcal{O}_{W_R(g\cdot f)} = (g \cdot \mathcal{B}) / \mathcal{I}_{Z_R (g \cdot f)} = g \cdot (\mathcal{B}/\mathcal{I}_{Z_R(f)}) = g \cdot \cO_{W_R(f)}$. It follows that $\xi(g \cdot f) = g \cdot \xi(f)$.  
 This proves the lemma.
\end{proof}
Since $\crimp^d_\Bbbk(\bP^1_\Bbbk)^{\rm{un}}$ is Noetherian, it has finitely many irreducible components. Let $\zeta_1, \cdots \zeta_r$ be their generic points. We define $X \subset \cover_{2d,d}^{\rm{un}} (\bP^1_\Bbbk)$ to be the reduced closed subscheme whose underlying topological space is the closure of the finite set $\{\psi(\zeta_1), \cdots \psi(\zeta_r)\}$. Write $\rho_X: X \to  \crimp^d_\Bbbk(\bP^1_\Bbbk)^{\rm{un}}$ for the restriction of the $\SL_2$-equivariant proper surjective morphism $
    \rho: \cover_{2d,d}^{\rm{un}} (\bP^1_{\Bbbk}) \to \crimp^d_{\Bbbk} (\bP^1_{\Bbbk})^{\rm{un}}$.  
\begin{lem} \label{L: universal_homeo}
With notation defined above, 
\begin{enumerate}
    \item $|X|= \psi(|\crimp^d_\Bbbk(\bP^1_\Bbbk)^{\rm{un}}|)$;
    \item the restriction $\rho_X$ is an $\SL_2$-equivariant finite universal homeomorphism whose set-theoretic inverse map is given by $\psi$.  
\end{enumerate}   
\end{lem}
\begin{proof} 
    Let $x \in |X|$. Then $x$ is a specialization of $\psi(\zeta_i)$ for some $i$. 
    Choose a DVR $R$ and a morphism $\Spec R \to \cover_{2d,d}^{\rm{un}} (\bP^1_{\Bbbk})$
    whose generic point maps to $\psi (\zeta_i)$ and whose special point maps to $x$. By \Cref{L: lift_DVR}, we know that the special point of this family is the set-theoretic lift of its image under $\rho$, i.e., $x = \psi (\rho(x))$. Hence $|X| \subset \psi(|\crimp^d_\Bbbk(\bP^1_\Bbbk)^{\rm{un}}|)$. 

    On the other hand, let $f \in |\crimp^d_\Bbbk(\bP^1_\Bbbk)^{\rm{un}}|$. Then $f$ is a specialization of some generic point $\zeta_i$, and we realize this specialization via some DVR $R$. Since the map $\rho: \cover_{2d,d}^{\rm{un}} (\bP^1_{\Bbbk}) \to \crimp^d_{\Bbbk} (\bP^1_{\Bbbk})^{\rm{un}}$ is proper, the generic lift $\psi(\zeta_i)$ extends uniquely over $R$. The special point, by \Cref{L: lift_DVR}, is exactly $\psi(f)$. Since $|X|$ is closed and the generic point lies in $|X|$, we have $\psi(f) \in |X|$. Therefore $\psi(|\crimp^d_\Bbbk(\bP^1_\Bbbk)^{\rm{un}}|) \subset |X|$. This shows the equality in the first item. 

    It follows that $\rho_X: X \to \crimp^d_\Bbbk(\bP^1_\Bbbk)^{\rm{un}}$ is bijective.  Let $f \in \crimp^d_\Bbbk(\bP^1_\Bbbk)^{\rm{un}}$ be a point, and let $x = \psi(f) \in X$. Then the set-theoretic map defines a morphism $\Spec k(x)  \to X$. Together with $\rho_X$, we get homomorphisms of residue fields $k(f) \to k(x) \to k(f)$ whose composition is the identity on $k(f)$. Hence $k(f) \cong k(x) $.  In particular, $\rho_X$ is radicial and this implies that it is  universally injective by \cite[\href{https://stacks.math.columbia.edu/tag/0482}{Tag 0482}]{stacks-project}. Since $\rho_X$ is quasi-finite and proper, $\rho_X$ is finite and hence integral. Therefore, $\rho_X$ is a finite  universal homeomorphism. 

    Since the action of $\SL_2$ permutes  the irreducible components of $\crimp^d_{\Bbbk} (\bP^1_{\Bbbk})^{\rm{un}}$, it preserves  the set of generic points $\{\zeta_1,\cdots \zeta_r\}$, and hence the set of their image under $\psi$. Therefore, $|X|$ is $\SL_2$-invariant as the topological closure. Since  $X$ is given the reduced closed subscheme structure, the restriction of $\SL_2$-action on $\cover_{2d,d}^{\rm{un}} (\bP^1_{\Bbbk})$ factors through $X$ uniquely. So $X$ is $\SL_2$-invariant. Then the $\SL_2$-equivariance of $\rho_X$ follows from that of $\rho$.

\end{proof}

We consider two $\SL_2$-linearized line bundles on $\cover_{2d,d}^{\rm{un}}(\bP^1_\Bbbk)$: $\mathcal{L}:= \pi^* \mathcal{O}_{\Hilb_{\bP^1_\Bbbk}^{2d}}(1)$, where $\pi: \cover_{2d,d}^{\rm{un}}(\bP^1_\Bbbk) \to \Hilb_{\bP^1_\Bbbk}^{2d}$ denotes the structure map; and the determinant of the universal quotient bundle $\det Q_{\mathrm{univ}}$ (c.f. \Cref{lem: representability_of_H}). Set $\mathcal{M}_{a,b}:= \mathcal{L}^{\otimes a} \otimes (\det Q_{\mathrm{univ}})^{\otimes b}$, $a, b \in \bZ$.  

\begin{lem}\label{L: ampleness_range}  
$\mathcal{M}_{a,b}$ is an $\SL_2$-linearized ample line bundle on $\cover_{2d,d}^{\rm{un}}(\bP^1_\Bbbk)$, for any integers $b>0, a> bd$.      
\end{lem}
\begin{proof}
We denote $B := \Hilb_{\bP^1_\Bbbk}^{2d}$. 
Let $Z_{\rm{univ}} \subset \bP^1_\Bbbk \times B$ be the universal divisor of degree $2d$. Set $E = p_* \cO_{Z_{\rm{univ}}}$, where $p$ denotes the projection $\bP^1_\Bbbk \times B \to B$. Then $E$ is a locally free sheaf of $\cO_B$-algebras of rank $2d$ on $B$, and thus we have an inclusion $\cO_B \subset E$. We compute the quotient $\bar{E}:= E/\cO_B $ explicitly. 

Because $Z_{\rm{univ}}$ is given by the universal binary form of degree $2d$, its ideal sheaf is identified with $\cO_{\bP^1_\Bbbk}(-2d) \boxtimes \cO_{B}(-1)$. So we have an exact sequence on $\bP^1_\Bbbk \times B$, 
\[
0 \to \cO_{\bP^1_\Bbbk}(-2d) \boxtimes \cO_{B}(-1) \to \cO_{\bP^1_\Bbbk \times B} \to \cO_{Z_{\rm{univ}}} \to 0. 
\]
Pushing this forward to $B$, because $p_*(\cO_{\bP^1_\Bbbk}(-2d) \boxtimes \cO_{B}(-1)) \cong H^0(\bP^1_\Bbbk, \cO_{\bP^1_\Bbbk}(-2d)) \otimes  \cO_{B}(-1) =0$, and $R^1p_* (\cO_{\bP^1_\Bbbk \times B}) =0$, 
\[
0 \to \cO_{B} \to p_* \cO_{Z_{\rm{univ}}}= E  \to  H^1(\bP^1_\Bbbk, \cO_{\bP^1_\Bbbk}(-2d)) \otimes \cO_{B}(-1) \to 0.  
\]
Hence $\bar{E} \cong H^1(\bP^1_\Bbbk, \cO_{\bP^1_\Bbbk}(-2d)) \otimes \cO_{B}(-1)$, which is a locally free sheaf of rank $2d-1$ on $B$. On the relative Grassmannian $G = \Gr_d(E)$, there is a universal exact sequence $ 0 \to K_{\mathrm{univ}} \to \mathcal{O}_{Z_G} \to Q_{\mathrm{univ}} \to 0$ (c.f. \Cref{lem: representability_of_H}). Then the Grassmannian $\Gr_d(\bar{E})$ is exactly the closed subscheme of $G$ given by  the condition that the composition $\cO_G \to \mathcal{O}_{Z_G} \to Q_{\mathrm{univ}}$ vanishes. Since every subalgebra contains the unit, the closed immersion $\cover_{2d,d}^{\rm{un}}(\bP^1_\Bbbk) \hookrightarrow G$ factors through the closed subscheme $\Gr_d(\bar{E})$. Let $V = H^1(\bP^1_\Bbbk, \cO_{\bP^1_\Bbbk}(-2d))$. Then $\Gr_d(\bar{E}) \cong \Gr_d(V) \times B$. Hence restricting to $\Gr_d(\bar{E})$,  $Q_{\rm{univ}} \cong Q_{\Gr_d(V)} \boxtimes \cO_B(-1)$ under this isomorphism. 
Then $\det Q_{\rm{univ}} = \cO_{\Gr(V)} (1) \boxtimes \cO_B(-d)$. Therefore, $ \mathcal{L}^{\otimes a} \otimes (\det Q_{\mathrm{univ}})^{\otimes b}
\cong \cO_{\Gr(V)}(b) \boxtimes \cO_{B}(a-bd)$, which is ample whenever $b>0, a> bd$. Restricting further to $\cover_{2d,d}^{\rm{un}}(\bP^1_\Bbbk)$, we see that $\mathcal{M}_{a,b}$ is ample if $b>0, a> bd$.   
\end{proof}

Let $\iota: X \hookrightarrow \cover_{2d,d}^{\rm{un}}(\bP^1_\Bbbk)$ denote the $\SL_2$ equivariant closed immersion. We set $\mathcal{M}_{a,b}^X := \iota^* \mathcal{M}_{a,b}$. Since $X$ is a closed subscheme, $\mathcal{M}_{a,b}^X$ is ample whenever $b>0, a>bd$. Since $\mathcal{M}_{a,b}^X$ is $\SL_2$-linearized, it descends to an ample line bundle on $X/\SL_2$ which induces a numerical invariant $\mu_{a,b}^{X/\SL_2}$ on $X/\SL_2$ for $b>0, a>bd$. \Cref{L: universal_homeo} gives a representable universal homeomorphism $X/\SL_2 \to \mathcal{U}_{0,d}$. So by \Cref{P: general_statement}, $\mu_{a,b}^{X/\SL_2}$ is the pullback a unique numerical invariant $\mu_{a,b}$ on $\mathcal{U}_{0,d}$.

Suppose that $A \subset k[x]$ is a subalgebra of $k[x]$ of colength $d$. Let $\mathbb{G}_m$ act on $k[x]$ by $t \cdot x = t^m x$, where $m > 0$. For any nonzero polynomial $f = \sum_{i=0}^n c_i x^i \in k[x]$, we have $t \cdot f(x) = \sum_{i=0}^n c_i t^{mi} x^i$. Then the limit is given by the lowest (resp. highest) degree terms of elements of $A$ as $t \to 0$ (resp. $t \to \infty$). Thus,
\begin{align*}
A_0 &:= \lim_{t \to 0}t \cdot A = \mathrm{span}_k \{ x^l \mid l = \mathrm{ord}_0(f) \text{ for some nonzero} f \in A \}, \\
A_\infty & :=\lim_{t \to \infty} t \cdot A = \mathrm{span}_k \{ x^l \mid l = \deg (f) \text{ for some nonzero} f \in A \}. 
\end{align*}
Define
\[
 G_0(A) := \mathbb{N} \setminus \{ l \in \mathbb{N} ~|~ l = \mathrm{ord}_0(f) \text{ for some nonzero } f \in A\},
\]
and similarly, 
\[
G_\infty (A):= \mathbb{N} \setminus \{ l \in \mathbb{N} ~|~ l = \deg(f) \text{ for some nonzero } f \in A\}. 
\]
Then we have $|G_0(A)| \leq d$, and $|G_\infty(A)| =d$. Equivalently, the limit subalgebra $A_0$ (resp. $A_\infty$) is the monomial subalgebra spanned by the monomials  $x^l$ whose exponent $l \notin G_0$ (resp. $G_\infty$). 

Cover $\bP^1_k$ with two affine charts $U_1= \{ Y \neq 0\}$ and $U_2 := \{ X \neq 0\}$, with affine affine coordinates $x= X/Y$, and $y = Y/X$.
A crimping $(f: \bP^1_k \to C) \in \crimp^d_{\Bbbk} (\bP^1_{\Bbbk})^{\rm{un}} (k)$ determines
a subalgebra $A \subset k[x]$ of colength $|G_\infty (A)|$, and a subalgebra $B \subset k[y]$ of colength $|G_\infty (B)|$, so that on the overlap $U_1 \cap U_2$,  they are compatible under the canonical identification, $k[x, x^{-1}] \xrightarrow{\sim} k[y,y^{-1}], x\mapsto y^{-1}$, and moreover $|G_0 (A)| + |G_\infty (B)| = d$. 

Now consider a nondegenerate filtration $\eta: \Theta_k \to \mathcal{U}_{0,d}$ such that $\ev_1(\eta) =f$. Assume its associated graded point $\ev_0(\eta) = \lambda: B\bG_m \to \mathcal{U}_{0,d}$
is represented by  $f_0 \in \crimp^d_{\Bbbk}(\bP^1_{\Bbbk})^{\rm{un}}(k)$, together with a cocharacter $\gamma: \bG_m \to \SL_2$. Up to conjugation, we may assume  $\gamma(t) = \diag(t^r, t^{-r})$, for some $r>0$. Then $|\lambda| = r$.  In homogeneous coordinates, $\gamma(t) \cdot (X:Y) := (t^r X : t^{-r} Y)$. On the affine chart $U_1$ (resp. $U_2$), this action becomes $\gamma(t) \cdot x = t^{2r} x$ (resp. $\gamma(t) \cdot y = t^{-2r} y$). As $t \to 0$, the limit crimping $f_0 = \lim_{t\to 0} \gamma(t) \cdot f$ is given by two subalgebras $A_0 \subset k[x]$, and $B_{\infty} \subset k[y]$. Hence $Z_{f_0} = 2 |G_0 (A)| [0] + 2|G_\infty (B)|[\infty]$. This gives 
\[
\frac{\wt(\psi(\lambda)^*(\mathcal{L}))}{|\lambda|}= 2 ( |G_\infty (B)| -|G_0 (A)|),
\]
and 
\[
\frac{\wt(\psi(\lambda)^*(\det Q_{\mathrm{univ}}))}{|\lambda|} = 2 (  \sum_{l \in G_0(A)} l - \sum_{l \in G_\infty (B)} l).
\]


Therefore, 
\begin{equation}
  \mu_{a,b} (\eta) = \mu_{a,b}^{X/\SL_2}(\lambda) =   2a ( |G_\infty (B)| -|G_0 (A)|) + 2b (\sum_{l \in G_0(A)} l - \sum_{l \in G_\infty (B)} l).
\end{equation}

\begin{defn}
   A point of $|\mathcal{U}_{0,d}|$, represented by a crimping $(f: \bP^1_k \to C)$ over a field $k/\Bbbk$ is called \emph{$\Theta$-semistable} with respect to $\mu_{a,b}$, if for every field extension $L/k$, and every nondegenerate filtration $\eta: \Theta_L \to \mathcal{U}_{0,d}$ with $\ev_1(\eta) \cong f_L$, one has  $\mu_{a,b} (\eta)  \geq 0$.  
   In this case, we also say the corresponding geometrically  unibranch curve $C$ is  $\Theta$-semistable.
\end{defn}
\begin{thm}\label{T:genus_0}
        Let $b> 0, a> bd$. Then the numerical invariant $\mu_{a,b}$ defines a $\Theta$-stratification on $\mathcal{U}_{0,d}$. Moreover, the $\Theta$-semistable locus $\mathcal{U}_{0,d}^{\mathrm{ss}}$ admits a good moduli space that is proper over $\Bbbk$. 
    \end{thm}
    \begin{proof}
    When $b> 0, a> bd$, $\mathcal{M}_{a,b}^{X}$ is ample. Thus $\mu_{a,b}^{X/\SL_2}$ defines a $\Theta$-stratification on $X/\SL_2$ and $\Theta$-semistable locus denoted by $(X/\SL_2)^{\mathrm{ss}}$ admits a proper good moduli space \cite[Theorem 18.2.0.2]{HLModuliTheory}. Note that $\mu_{a,b}$ pulls back to $\mu_{a,b}^{X/\SL_2}$ via the finite representable universal homeomorphism $X/\SL_2 \to \mathcal{U}_{0,d}$, and $(X/\SL_2)^{\mathrm{ss}} \to (\mathcal{U}_{0,d})^{\mathrm{ss}}$ is again a finite representable universal homeomorphism. Hence it follows from \Cref{P: general_statement} that $\mu_{a,b}$ defines a $\Theta$-stratification \footnote{Here we define everything over characteristic $0$, any weak $\Theta$-stratification is a $\Theta$-stratification \cite[Corollary 2.1.9]{halpernleistner2022structureinstabilitymodulitheory} } on $\mathcal{U}_{0,d}$, and $(\mathcal{U}_{0,d})^{\mathrm{ss}}$ admits a separated good moduli space. 

    Moreover, because $(X/\SL_2)^{\mathrm{ss}}$ satisfies the existence part of the valuative criterion for properness, and  $(X/\SL_2)^{\mathrm{ss}} \to (\mathcal{U}_{0,d})^{\mathrm{ss}}$ is a finite representable universal homeomorphism, it follows  that $(\mathcal{U}_{0,d})^{\mathrm{ss}}$ also satisfies the existence part of the valuative criterion for properness. Therefore $(\mathcal{U}_{0,d})^{\mathrm{ss}}$ admits a proper good moduli space \cite[Theorem A]{Alper_2023}. This completes the proof.
    \end{proof}

\printendnotes
\printbibliography

\end{document}